\documentclass[11pt,reqno]{amsart}
\usepackage{mathrsfs,amssymb,tikz}
\usepackage[all]{xy}
\usepackage[colorlinks]{hyperref}

\title[Hodge numbers in positive characteristic]{Hodge numbers are not derived invariants \\ in positive characteristic}
\author[N.~Addington]{Nicolas Addington}
\address{Nicolas Addington \\
Department of Mathematics \\
University of Oregon \\
Eugene, Oregon 97403 \\
USA}
\email{adding@uoregon.edu}

\author[D.~Bragg]{Daniel Bragg}
\address{Daniel Bragg \\
Department of Mathematics \\
University of Utah \\
155 S 1300 E \\
Salt Lake City, Utah 84112 \\
USA}
\email{bragg@math.utah.edu}

\address{Alexander Petrov \\
Max-Planck-Institut f\"ur Mathematik \\
Vivatsgasse 7 \\
53111 Bonn \\
Germany}
\email{alexander.petrov.57@gmail.com }

\newcommand \E {\mathrm E}
\renewcommand \H {\mathrm H}
\newcommand \R {\mathrm R}

\newcommand \bC {\mathbf C}
\newcommand \bF {\mathbf F}
\newcommand \bP {\mathbf P}
\newcommand \bQ {\mathbf Q}

\newcommand \bZ {\mathbf Z}

\newcommand \cI {\mathscr I}
\newcommand \cO {\mathscr O}
\newcommand \cX {\mathscr X}

\DeclareMathOperator \HH {HH}
\DeclareMathOperator \Char {char}
\DeclareMathOperator \Pic {Pic}
\DeclareMathOperator \red {red}
\DeclareMathOperator \Aut {Aut}
\DeclareMathOperator \Br {Br}
\DeclareMathOperator \Gr {Gr}
\DeclareMathOperator \coker {coker}
\DeclareMathOperator \Spec {Spec}
\DeclareMathOperator \Hom {Hom}
\DeclareMathOperator \Ext {Ext}
\DeclareMathOperator \sPic {Pic}
\newcommand \piOneTop {\pi_1^\textup{top}}
\newcommand \piOneEt {\pi_1^\textup{\'et}}
\DeclareMathOperator \dR {dR}
\DeclareMathOperator \TR {TR}
\DeclareMathOperator \Tor {Tor}
\DeclareMathOperator \NS {NS}
\DeclareMathOperator \DM {DM}
\DeclareMathOperator \rank {rank}
\DeclareMathOperator \Hdg {Hdg}
\DeclareMathOperator \td {td}
\DeclareMathOperator \ch {ch}
\DeclareMathOperator{\cHom}{\cH\!\it om}

\newcommand \cH {\mathscr H}
\newcommand \hA {\widehat A}
\newcommand \hB {\widehat B}
\newcommand \fA {\mathfrak A}
\newcommand \fB {\mathfrak B}
\newcommand \fX {\mathfrak X}
\newcommand \fZ {\mathfrak Z}
\DeclareMathOperator \cris {cris}
\usepackage[bbgreekl]{mathbbol}
\newcommand \Prism {\displaystyle \mathbb \Delta}
\newcommand \oPrism {\overline \Prism}

\newtheorem{theorem}{Theorem}[section]
\newtheorem{lemma}[theorem]{Lemma}
\newtheorem{proposition}[theorem]{Proposition}
\newtheorem{corollary}[theorem]{Corollary}
\theoremstyle{definition}
\newtheorem{remark}[theorem]{Remark}
\newtheorem{remarks}[theorem]{Remarks}
\newtheorem{questions}[theorem]{Questions}
\newtheorem{convention}[theorem]{Convention}

\numberwithin {equation} {section}

\usepackage{array}
\newcommand{\PreserveBackslash}[1]{\let\temp=\\#1\let\\=\temp}
\newcolumntype{C}[1]{>{\PreserveBackslash\centering}p{#1}}

\begin{document}

\maketitle

\begin{abstract}
We study a pair of Calabi--Yau threefolds $X$ and $M$, fibered in non-principally polarized Abelian surfaces and their duals, and an equivalence $D^b(X) \cong D^b(M)$, building on work of Gross, Popescu, Bak, and Schnell.  Over the complex numbers, $X$ is simply connected while $\pi_1(M) = (\bZ/3)^2$.  In characteristic 3, we find that $X$ and $M$ have different Hodge numbers, which would be impossible in characteristic 0.

In an appendix, we give a streamlined proof of Abuaf's result that the ring $\H^*(\cO)$ is a derived invariant of complex threefolds and fourfolds.

A second appendix by Alexander Petrov gives a family of higher-dimensional examples to show that $h^{0,3}$ is not a derived invariant in any positive characteristic.
\end{abstract}


\section{Introduction}

What cohomological invariants are preserved by equivalences of derived categories of coherent sheaves?  We consider a smooth proper variety $X$ over a field $k$.  Certainly the Hochschild numbers $hh_i = \dim \HH_i(X)$ are preserved.  If $\Char k = 0$ or $\Char k > \dim X$ then these are sums of Hodge numbers $h^{i,j} = \dim H^j(\Omega^i_X)$, namely,
\[ hh_i = \sum_j h^{j,j-i}. \]
More subtly, Antieau and Vezzosi showed that the same equality holds when $\Char k = \dim X$ \cite[Thm.~1.3]{antieau-vezzosi}.

Popa and Schnell showed that in characteristic zero, the first Betti number is derived invariant \cite{ps}, which implies that the Hodge numbers $h^{1,0}=h^{0,1}$ are derived invariant.  Achter, Casalain-Martin, Honigs, and Vial extended the result on $b_1$ to positive characteristic \cite[Thm.~A.1]{honigs2}, but there it no longer implies the result on $h^{1,0}$ and $h^{0,1}$, as we will see below.  Abuaf showed that when $\Char k = 0$ and $\dim X \le 4$, the cohomology \emph{ring} $\H^*(\cO_X)$ is a derived invariant \cite[Thm.~1.3(4)]{abuaf}; in Appendix \ref{app:abuaf} we give a streamlined account of his proof.

Taken together, these results imply that all Hodge numbers are derived invariants in characteristic zero when $\dim X \le 3$, and a conjecture of Orlov \cite[Conj.~1]{orlov_motives} would imply that this continues in higher dimensions.  In positive characteristic, Antieau and Bragg showed that Hodge numbers are derived invariants when $\dim X \le 2$ \cite[Thm.~1.3(1)]{antieau_bragg}, and the Euler characteristics $\chi(\Omega^i_X)$ are derived invariants when $\dim X = 3$ \cite[Thm.~5.33]{antieau_bragg}.  But in this paper we show that Hodge numbers of threefolds are not derived invariants in positive characteristic:

\begin{theorem} \label{main_thm}
There are smooth projective threefolds $X$ and $M$ defined over $\bar{\bF}_3$, with Hodge numbers $h^{i,j} = h^j(\Omega^i)$ as shown,
\[
\arraycolsep=-1pt
\begin{array}{ccccccc}
\\[1ex]
            &         &         & h^{3,3} &         &         &         \\
            &         & h^{3,2} &         & h^{2,3} &         &         \\
            & h^{3,1} &         & h^{2,2} &         & h^{1,3} &         \\
    h^{3,0} &         & h^{2,1} &         & h^{1,2} &         & h^{0,3} \\
            & h^{2,0} &         & h^{1,1} &         & h^{0,2} &         \\
            &         & h^{1,0} &         & h^{0,1} &         &         \\
            &         &         & h^{0,0} &         &         &
\end{array}
\arraycolsep=4pt
\begin{array}{ccccccccccccccccc}
&   &   &   &   & X &   &   &   &            &   &   &   & M &   &   &   \\[\bigskipamount]
&   &   &   &   & 1 &   &   &   &            &   &   &   & 1 &   &   &   \\
&   &   &   & 0 &   & 0 &   &   &            &   &   & 1 &   & 1 &   &   \\
&   &   & 0 &   & 7 &   & 1 &   &            &   & 1 &   & 7 &   & 0 &   \\
& = & 1 &   & 8 &   & 8 &   & 1 & \text{vs.} & 1 &   & 6 &   & 6 &   & 1 \\
&   &   & 1 &   & 7 &   & 0 &   &            &   & 0 &   & 7 &   & 1 &   \\
&   &   &   & 0 &   & 0 &   &   &            &   &   & 1 &   & 1 &   &   \\
&   &   &   &   & 1 &   &   &   &            &   &   &   & 1 &   &   &
\end{array}
\]
and an $\bar\bF_3$-linear exact equivalence $D^b(X) \cong D^b(M)$.
\end{theorem}
Our $X$ and $M$ lift to characteristic zero, where both have Hodge diamond
\begin{equation} \label{eq:1661}
\arraycolsep=4pt
\begin{array}{ccccccccccccccc}
  &   &   & 1 &   &   &   \\
  &   & 0 &   & 0 &   &   \\
  & 0 &   & 6 &   & 0 &   \\
1 &   & 6 &   & 6 &   & 1 \\
  & 0 &   & 6 &   & 0 &   \\
  &   & 0 &   & 0 &   &   \\
  &   &   & 1\makebox[0pt]{ ,}
\end{array}
\end{equation}
but their Hodge numbers jump in different ways when we reduce mod 3.

Our construction is motivated by work of Bak \cite{bak} and Schnell \cite{schnell}, who produced the first example of a pair of derived equivalent complex threefolds with different fundamental groups: one was a Calabi--Yau threefold constructed by Gross and Popescu in \cite[\S6]{grpo}, fibered in $(1,8)$-polarized Abelian surfaces over $\bP^1$, and the other was the dual Abelian fibration.  We initially hoped that that pair would have interesting properties in characteristic 2, but writing the same equations over $\bar\bF_2$, we found that no choice of parameters yielded smooth threefolds.  Instead, for our threefold $X$ we take a different example from the same paper of Gross and Popescu \cite[\S4]{grpo}, fibered in $(1,6)$-polarized Abelian surfaces, and we find that it behaves well over $\bar\bF_3$.  Then over both $\bC$ and $\bar\bF_3$ we construct $M$ as a compactified relative Picard scheme.  As a side benefit, we produce another pair of derived equivalent complex threefolds with different fundamental groups.

We face two issues that were not present in the $(1,8)$ example.  First, the Abelian fibration $X \to \bP^1$ has some reducible fibers, so more care is required to construct the dual Abelian fibration $M \to \bP^1$.  Second, in the $(1,8)$ example, $M$ was isomorphic to the quotient of $X$ by a free action of $\bZ/8 \times \mu_8$, but in the $(1,6)$ example, the analogous action of $\bZ/6 \times \mu_6$ on $X$ is not free.  The subgroup $\bZ/3 \times \mu_3$ does act freely, however, and this is enough to make $\H^1(M, \bZ/3) \ne 0$.  Over $\bar\bF_3$, the Artin--Schreier exact sequence
\[ \xymatrix@R=0pt{
0 \ar[r] & \bZ/3 \ar[r] & \cO_M \ar[r] & \cO_M \ar[r] & 0 \\
& & f \ar@{|->}[r] & f^3 - f
} \]
yields an injection $\H^1(M,\bZ/3) \hookrightarrow \H^1(\cO_M)$, explaining the jump in $h^{0,1}$.  Lam also noticed this possibility and discussed it in \cite[\S9]{lam}.

We prove Theorem \ref{main_thm} in \S\ref{sec:construction}--\S\ref{sec:finish_the_proof}.  We rely on the computer algebra system Macaulay2 \cite{M2} at several points: to check that Gross and Popescu's equations for $X$ give a smooth threefold over $\bar\bF_3$, to compute its Hodge numbers, to show that the map $X \to \bP^1$ still has smooth fibers, and to analyze the fixed locus of the $\bZ/6 \times \mu_6$ action.  Our construction of $M$ as a moduli space of sheaves does not yield equations that can go into a computer, so we study $M$ by hand.

In \S\ref{sec:crystalline} we compute the crystalline and Hodge--Witt cohomology of $X$ and $M$, which clarifies what is going on with the Hodge numbers.  The free parts of these cohomology groups are the same for $X$ and $M$, but the torsion parts appear in different places, reminiscent of the behavior of singular cohomology of complex Calabi--Yau threefolds seen in \cite{br_not_inv}.  The role that topological K-theory plays for derived equivalences over $\bC$ is played here by Hessolholt's $\TR$ invariants, sometimes called topological restriction homology.

\begin{remarks} \label{rem:intro} \
\begin{enumerate}
\addtolength \itemsep \smallskipamount

\item Van der Geer and Katsura asked whether a Calabi--Yau threefold in positive characteristic can have $h^{1,0} \ne 0$ or $h^{2,0} \ne 0$ in \cite[\S7]{vdg-katsura}.  Our $X$ shows that $h^{2,0} \ne 0$ is possible.  The quotient $X/\mu_3$ shows that $h^{1,0} \ne 0$ is possible as well, exactly like Lam's example of a quintic threefold modulo $\mu_5$ in characteristic 5 \cite[Thm.~1.2]{lam}.  Like that example, $X/\mu_3$ lifts to characteristic zero, giving another counterexample to a conjecture of Joshi \cite[Conj.~7.7.1]{joshi2}.

\item Our $M$ is weakly ordinary by Proposition \ref{prop:weakly_ordinary_derived_invariant}, and has 3-torsion in its Picard group by Proposition \ref{prop:pictau}(a), answering a question asked by Patakfalvi of the second author.

\item The isogeny class of the reduction $(\sPic^0)_{\red}$ is a derived invariant by \cite[Thm.~A.1]{honigs2}, but our varieties have $\sPic^0_X = 0$ and $\sPic^0_M = \mu_3$, showing that it is necessary to take the reduction.  On the other hand, Rouquier showed that $\Aut^0 \ltimes \sPic^0$ is a derived invariant \cite[Prop.~9.45]{huybrechts_fm}, and in our example the action of $\bZ/6 \times \mu_6$ on $X$ gives a copy of $\mu_3$ in $\Aut^0_X$, while the fact that $h^0(T_M) = h^{2,0}(M) = 0$ gives $\Aut^0_M = 0$.

\item Although $X$ and $M$ have different Hodge numbers, we do not expect them to contradict Orlov's conjecture mentioned earlier \cite[Conj.~1]{orlov_motives}, which says that they should have isomorphic rational Chow motives.

\item Our $M$ raises questions about the right definition of a Calabi--Yau threefold in positive characteristic.  Of course a Calabi--Yau threefold should have trivial canonical bundle, and should not be an Abelian threefold or a product of an elliptic curve and a K3 surface (or non-classical Enriques surface).  We suggest that a Calabi--Yau threefold should be defined (in any characteristic) as one with $\omega \cong \cO$ and $b_1 = 0$.  Requiring $\pi_1 = 0$ would be too restrictive, even in characteristic zero, as discussed in \cite{br_not_inv}.  Requiring $h^{0,*} = 1, 0, 0, 1$, as many authors do, is equivalent to $b_1 = 0$ in characteristic zero, but stronger in characteristic $p$, as our $M$ demonstrates, and we feel that it is also too restrictive: our $M$ lifts to a Calabi--Yau threefold, so it should qualify as one.  The situation is comparable to that of Enriques surfaces, where $h^{0,1}$ can jump in characteristic 2.  All this jumping occurs because Hodge numbers in characteristic $p$ are influenced by torsion in cohomology, unlike in characteristic zero; to obtain a useful definition in all characteristics, Betti numbers seem more robust than Hodge numbers.

\end{enumerate}
\end{remarks}

\begin{questions} \ 
\begin{enumerate}
\addtolength \itemsep \smallskipamount

\item Antieau and Bragg remarked in \cite[Rmk.~5.34]{antieau_bragg} that at least 14 independent linear relations hold among the 16 Hodge numbers of derived equivalent threefolds in characteristic $p \ge 3$.  Our example shows that at most 15 relations hold.  We wonder whether the true number is 14 or 15.

\item Our $X$ deserves further study over $\bC$.  It is a small resolution of an intersection of two special cubic fourfolds, similar to the ones studied by Calabrese and Thomas \cite{calabrese-thomas}.  The reducible fibers of the Abelian fibration $X \to \bP^1$ seem to be unions of two sextic elliptic ruled surfaces, of the kind that appeared in work of Addington, Hassett, Tschinkel, and V\'arilly--Alvarado \cite{ahtva}.  This would imply that the special cubics have discriminant 18, so there should be another Calabi--Yau threefold $Z$ fibered in K3 surfaces of degree 2, a 3-torsion Brauer class $\alpha \in \Br(Z)$, and a derived equivalence $D^b(X) \cong D^b(Z,\alpha)$.

\item Still over $\bC$, it would be interesting to get hold of the Brauer classes on $X$ that come from $\H_1(M,\bZ) = \pi_1(M) = (\bZ/3)^2$ via the exact sequence
\[ 0 \to \H_1(M,\bZ)  \to \Br(X) \to \Br(M) \to 0 \]
discussed in \cite{br_not_inv}.  Do they account for the whole Brauer group of $X$, as Gross and Pavanelli proved in the $(1,8)$-polarized example \cite{grpa}?  One might even try to use these Brauer classes to obstruct rational points over number fields.

\end{enumerate}
\end{questions}

\pagebreak 
In Appendix \ref{app:petrov}, Alexander Petrov adapts a construction from his paper \cite{petrov} to obtain, for any prime $p$, a pair of varieties $X_1$ and $X_2$ over $\bar\bF_p$ with $D^b(X_1) \cong D^b(X_2)$ but $h^{0,3}(X_1) \ne h^{0,3}(X_2)$.  His varieties, like ours, are dual Abelian fibrations, and the equivalence is a family version of Mukai's equivalence, but otherwise the flavor is quite different: his fibrations are isotrivial, there is no torsion in crystalline cohomology, and they do not lift to characteristic zero.  The base of his fibration is 4-dimensional, and the fibers are at least 4-dimensional. 

\subsection*{Acknowledgements}
We began this project at the conference ``Derived categories and geometry in positive characteristic'' in Warsaw in July 2019; we thank the organizers and IMPAN for their hospitality, and Ben Antieau for stimulating initial discussions.  We thank Mark Gross for advice on \cite{grpo}, Richard Thomas and Adrian Langer for advice on framed sheaves, and Ben Young for computer time.  Addington was supported by NSF grant no.\ DMS-1902213.  Bragg was supported by NSF grant no.\ DMS-1902875.


\section{Construction of \texorpdfstring{$X$}{X} and calculation of its Hodge numbers} \label{sec:construction}

In \cite[\S4]{grpo}, Gross and Popescu studied a complete family of Abelian surfaces in $\bP^5$ parametrized by a smooth quadric threefold $Q$, and the Calabi--Yau threefolds obtained by restricting this family to lines $L \subset Q$.  We begin by reviewing their construction over $\bC$, then analyze what happens when we carry out the same construction over $\bar\bF_3$. \bigskip

First, let $\bZ/6$ act on $\bP^5$ by powers of
\[ \sigma\colon (x_0 : x_1 : \dotsb : x_5) \mapsto (x_5 : x_0 : \dotsb : x_4). \]
Let $\mu_6$ act on $\bP^5$ by powers of
\[ \tau\colon (x_0 : x_1 : \dotsb : x_5) \mapsto (x_0 : \xi^{-1} x_1 : \dotsb : \xi^{-5} x_5), \]
where $\xi$ is a primitive sixth root of unity.  Observe that $\sigma \circ \tau$ and $\tau \circ \sigma$ agree up to a power of $\xi$, so $\sigma$ and $\tau$ commute as automorphisms of $\bP^5$.\footnote{Some readers may be interested to know that this action of $\bZ/6 \times \mu_6$ is related to the Schr\"odinger representation of the Heisenberg group.  For details see Gross and Popescu's paper, or \cite[Ch.~6]{birkenhake_lange} for a textbook account.}

The space of cubic polynomials invariant under the action of $\sigma^2$ and $\tau^2$ is 8-dimensional, spanned by
\begin{align*}
f_0 &= x_0^3 + x_2^3 + x_4^3 \\
f_1 &= x_1^2 x_4 + x_3^2 x_0 + x_5^2 x_2 \\
f_2 &= x_1 x_2 x_3 + x_3 x_4 x_5 + x_5 x_0 x_1 \\
f_3 &= x_0 x_2 x_4
\end{align*}
and $\sigma f_0, \dotsc, \sigma f_3$.

\pagebreak 
Consider the rational map
\[ \phi\colon \bP^5 \dashrightarrow \Gr(2,4) \subset \bP^5 \]
given by the $2 \times 2$ minors of the matrix
\begin{equation} \label{matrix_formerly_known_as_M}
\begin{pmatrix}
f_0 & f_1 & f_2 & f_3 \\
\sigma f_0 & \sigma f_1 & \sigma f_2 & \sigma f_3
\end{pmatrix}.
\end{equation}
This takes values in a linear section $Q \subset \Gr(2,4)$ because of the relation
\[ f_0 \cdot \sigma f_3 - f_3 \cdot \sigma f_0 = f_2 \cdot \sigma f_1 - f_1 \cdot \sigma f_2. \]
Let $\cX \subset \bP^5 \times Q$ be the graph of $\phi$; then $\pi_2\colon \cX \to Q$ is the desired family of Abelian surfaces.

Given a line $L \subset Q$, the Abelian-fibered threefold $X := \pi_2^{-1}(L)$ will turn out to be Calabi--Yau.  The variety of lines on $Q$ is isomorphic to $\bP^3$, and for a point
\[ p = (p_0 : p_1 : p_2 : p_3) \in \bP^3, \]
we can describe the line $L$ and threefold $X$ explicitly.  Let
\[ f_p = p_0 f_0 + p_1 f_1 + p_2 f_2 + p_3 f_3, \]
and let $Y \subset \bP^5$ be the complete intersection cut out by $f_p$ and $\sigma f_p$.  The line $L$ is $\phi(Y) \subset Q$.  Even more explicitly, the six sextics given by the $2 \times 2$ minors of \eqref{matrix_formerly_known_as_M} span a 5-dimensional space, of which three vanish on $Y$; the two remaining sextics give a rational map $Y \dashrightarrow \bP^1 = L$.  Then $X \subset Y \times \bP^1$ is the graph of this rational map, with projections
\[ \xymatrix{
X \ar[r]^-{\pi_2} \ar[d]_-{\pi_1} & \bP^1 \\
Y.
} \]
Our notation mostly matches Gross and Popescu's, except that our $Y$ is their $V_{6,p}$, and our $X$ is their $V^1_{6,p}$.

Gross and Popescu proved in \cite[Thm.~4.10]{grpo} that if the point $p \in \bP^3$ is chosen generically, then $Y$ has 72 ordinary double points and the map $\pi_1\colon X \to Y$ is a small resolution of singularities, so $X$ is a smooth Calabi--Yau threefold.  They calculated the Hodge diamond of $X$ to be the one given in the introduction \eqref{eq:1661}.  We will end up re-confirming these facts for a particular choice of $p \in \bP^3$.\bigskip

Now given a point $p \in \bP^3$ defined over $\bar\bF_3$ (or indeed any field), we can use the same equations to get a complete intersection $Y \subset \bP^5$, a rational map $Y \dashrightarrow \bP^1$, and its graph $X \subset Y \times \bP^1$.

For definiteness we will choose particular values of $p$ over $\bar\bF_3$ and $\bC$.  We expect that the set of $p$ for which the claims in this paper hold is Zariski open, but to prove this would require substantially more work.  Experimentally, we find that they hold for a few $p$ defined over $\bF_9$, for most $p$ defined over $\bF_{27}$, for nearly all $p$ defined over $\bF_{81}$ and higher, but unfortunately not for any $p$ defined over $\bF_3$.  We also find that they hold for nearly all $p$ defined over $\bQ$.\footnote{We will confine our attention to characteristic 0 and 3, but the behavior in characteristic $\ge 5$ seems to be the same as the behavior in characteristic 0, while the behavior in characteristic 2 is very different: for example, the singularities of $Y$ are not isolated, and $X$ is not smooth.}

\begin{convention} \ 
\begin{enumerate}
\addtolength \itemsep \smallskipamount

\item When we discuss $X$ and $Y$ over $\bar\bF_3$, we take
\[ p = (1 : \alpha+1 : -\alpha : -\alpha) \in \bP^3, \]
where $\alpha$ is a root of $x^2 + 2x + 2$.  (This is the Conway polynomial used to construct $\bF_9$ in most computer algebra systems.)

\item When we discuss $X$ and $Y$ over $\bC$, we take
\[ p = (1 : i : 1-i : 1-i) \in \bP^3. \]
We point out that if we regard this choice as being defined over $\bZ[i]$, then its reduction modulo 3 gives back our earlier choice over $\bF_9$, because $i-1$ is a root of $x^2+2x+2$.

\item When we do not specify the field, we are making a claim about both of these.
\end{enumerate}
\end{convention}

\begin{proposition} \label{sec2_prop} \
\begin{enumerate}
\addtolength \itemsep \smallskipamount

\item $Y$ has isolated singularities.

\item The map $\pi_1\colon X \to Y$ is a small resolution of singularities: it is an isomorphism over the smooth part of $Y$, the fibers over the singular points are 1-dimensional, and $X$ is smooth.  Thus the canonical bundle $\omega_X \cong \cO_X$.

\item $\R \pi_{1*} \cO_X = \cO_Y$, and in particular, $h^*(\cO_X) = 1, 0, 0, 1$.

\item Over $\bC$ we have $h^*(T_X)= 0, 6, 6, 0$. \\ Over $\bar\bF_3$ we have $h^*(T_X) = 1, 8, 7, 0$.

\end{enumerate}

\end{proposition}

\begin{proof}
Our proof relies on the computer algebra system Macaulay2 \cite{M2}, using the code in the ancillary file \verb|verify.m2|.  For readers who prefer Magma \cite{magma}, we also provide \verb|verify.magma|, which cannot do everything we need, but is faster at the things it can do. \smallskip

(a) Let $Z \subset Y$ be the subscheme cut out by the $2 \times 2$ minors of the Jacobian matrix
\begin{equation} \label{jacobian}
J = \begin{pmatrix}
\partial f_p/\partial x_0 & \dots & \partial f_p/\partial x_5 \\
\partial(\sigma f_p)/\partial x_0 & \dots & \partial(\sigma f_p)/\partial x_5 \\
\end{pmatrix}.
\end{equation}
The singular locus of $Y$ is the reduction $Z_\text{red}$.  With Macaulay2 we find that $Z$ is zero-dimensional, so $Y$ has isolated singularities. \smallskip

(b) With Macaulay2 we find that the map $\pi_1\colon X \to Y$ is smooth of relative dimension zero away from $Z \subset Y$ using the Jacobian criterion.  The equations cutting out $X \subset Y \times \bP^1$ are linear in the $\bP^1$ variables, so $\pi_1\colon X \to Y$ has degree 1 away from $Z$, so it is an isomorphism there.  Next we find that $\pi_2\colon X \to \bP^1$ is smooth along $\pi_1^{-1}(Z)$, so $X$ is smooth at every point.

Because $Y$ is an intersection of two cubics in $\bP^5$, the adjunction formula gives $\omega_Y \cong \cO_Y$, so the line bundles $\omega_X$ and $\cO_X$ agree on the open set $U \subset X$ where $\pi_1$ is an isomorphism.  The complement $X \setminus U$ has codimension 2, so $\omega_X \cong \cO_X$. \smallskip

(c) Macaulay2 gives a free resolution of $\cO_X$ as a sheaf on $\bP^5 \times \bP^1$:
\begin{align} 
0 \to \cO(-9,-1)^2 &\to \cO(-8,-1)^9 \notag \\
&\to \cO(-7,-1)^{12} \oplus \cO(-6,-1)^8 \notag \\
& \to \cO(-6,-1)^3 \oplus \cO(-5,-1)^{18} \oplus \cO(-6,0) \label{eq:res} \\
& \to \cO(-4,-1)^6 \oplus \cO(-3,-1)^2 \oplus \cO(-3,0)^2 \notag \\
& \hspace{2in} \to \cO \to \cO_X \to 0. \notag
\end{align}
The terms are $\pi_{1*}$-acyclic, so we can compute $\R\pi_{1*} \cO_X$ by applying $\pi_{1*}$ termwise.  The terms of the form $\cO(*, -1)$ drop out, leaving the Koszul resolution of $\cO_Y$:
\[ 0 \to \cO_{\bP^5}(-6) \to \cO_{\bP^5}(-3)^2 \to \cO_{\bP^5} \to \cO_Y \to 0. \]
So we have $\R\pi_{1*} \cO_X = \cO_Y$, and thus $\H^*(\cO_X) = \H^*(\cO_Y)$.  We compute $\H^*(\cO_Y)$ using the Koszul resolution. \smallskip

(d) It is infeasible for the computer to find $\H^*(T_X)$ directly.  Instead we will produce a complex of sheaves on $Y$ that is quasi-isomorphic to $\R\pi_{1*} T_X$, and then apply the hypercohomology spectral sequence.

First we study the cotangent complex $L_Y$ of $Y$, or rather the tangent complex $L_Y^\vee$.  It is a 2-term complex of vector bundles
\[ \underline{T_{\bP^5}|_Y} \to \cO_Y(3)^2, \]
where the underlined term is in degree zero.  Or if we replace $T_{\bP^5}$ with its Euler resolution, then $L_Y^\vee$ is quasi-isomorphic to the three-term complex
\[ \cO_Y \xrightarrow{\left(\begin{smallmatrix} x_0 \\ \vphantom{\int\limits^x}\smash{\vdots} \\ x_5 \end{smallmatrix} \right)} 
\underline{\cO_Y(1)^6} \xrightarrow{J} \cO_Y(3)^2, \]
where $J$ is the Jacobian matrix that appeared in \eqref{jacobian}.  In the proof of (a) above we studied the support of $\coker J$, which we called $Z \subset Y$.  Using Macaulay2, we find that the rank of $J$ never drops to 0 on $Y$, so $\coker J$ is just $\cO_Z$.

Next we study the tangent complex of the morphism $\pi_1\colon X \to Y$.  It is the cone on $T_X \to \pi_1^* L_Y^\vee$, which is a 3-term complex of vector bundles
\[ 0 \to T_X \to \underline{\pi_1^* T_{\bP^5}} \to \pi_1^* \cO_Y(3)^2 \to 0. \]
This complex is exact on the open set $U \subset X$ where $\pi_1$ is an isomorphism.  The complement $X \setminus U$ has codimension 2, hence depth 2 because $X$ is smooth and thus Cohen--Macaulay, so by a theorem of Buchsbaum and Eisenbud \cite[Thm.~20.9]{eisenbud}, the complex is exact except perhaps at the right.  And indeed we see that the cokernel of the last map is
\[ \pi_1^* \coker J = \pi_1^* \cO_Z = \cO_{\pi_1^{-1}(Z)}. \]

With Macaulay2 we find that $\pi^{-1}(Z)$ is of the form $Z' \times \bP^1$ for a subscheme $Z' \subset Y$ with the same support as $Z$.  In fact over $\bC$ we have $Z' = Z$, but over $\bar\bF_3$ they are different: $Z$ has length 144, while $Z'$ has length 72.

Thus $T_X$ is quasi-isomorphic to the complex
\[ \underline{\pi_1^* T_{\bP^5}} \to \pi_1^* \cO_Y(3)^2 \to \cO_{Z' \times \bP^1}. \]
The terms are acyclic for $\pi_{1*}$, as we see using the adjunction formula and the fact that $\R\pi_{1*} \cO_Y = \cO_X$.  So $\R\pi_{1*} T_X$ is quasi-isomorphic to the complex
\[ \underline{T_{\bP^5}|_Y} \to \cO_Y(3)^2 \to \cO_{Z'}. \]
We remark that over $\bC$, where $Z' = Z$, this means that $\R\pi_{1*} T_X$ is just the tangent sheaf $T_Y := \mathcal \H^0(L_Y^\vee)$, whose cohomology can in fact be computed with Macaulay2, although it is fairly slow.

Replacing $T_{\bP^5}$ with its Euler resolution, we get a complex
\[ \cO_Y \to \underline{\cO_Y(1)^6} \xrightarrow{J} \cO_Y(3)^2 \to \cO_{Z'}, \]
where the last map is the composition of the map $\cO_Y(3)^2 \twoheadrightarrow \coker J = \cO_Z$ and the map $\cO_Z \twoheadrightarrow \cO_{Z'}$.  We will compute $\R\Gamma(T_X) = \R\Gamma(\R\pi_{1*} T_X)$ using the hypercohomology spectral sequence, which for a complex $A^\bullet$ and a left-exact functor $F$ reads
\[ \E_1^{i,j} = \R^j F(A^i) \Rightarrow \R^{i+j}F(A^\bullet). \]
With $F = \Gamma$ applied to our complex, the $\E_1$ page is
\[ \begin{array}{c|ccccccc}
j=3 & k \\
j=2 & 0 \\
j=1 & 0 \\
j=0 & k & \to & k^{36} & \to & k^{108} & \to & k^{72} \\
\hline
& i=-1 && i=0 && i=1 && i=2,
\end{array} \]
where $k$ is either $\bC$ or $\bar\bF_3$.  In the bottom row, the map $k \to k^{36}$ must be injective: otherwise the kernel would contribute to $\H^{-1}(T_X)$.  It is straight\-forward to compute the rank of the map $k^{36} \to k^{108}$ map using Macaulay2: it is 35 over $\bC$ or 34 over $\bar\bF_3$.

The map $k^{108} \to k^{72}$ is more subtle.  The computer works on the coordinate ring of $Y$,  that is, on $S_Y := k[x_0,\dotsc,x_5]/(f_p,\sigma f_p)$.  While the Hilbert polynomial of the module $N := \coker J \otimes S_Y/I_{Z'}$ is constantly 72, the Hilbert function of $N$ does not stabilize immediately; over $\bC$ it only becomes constantly 72 in degree $d \ge 1$, and over $\bar\bF_3$, in degree $d \ge 2$.  Thus we need to compose the natural map $S(3)^2 \to N$ with a map $N \to N(d)$ given by multiplication with some homogeneous $g \in S_Y$ of degree $d \ge 2$ that does not vanish on $Z'$.  With this done, we find that the rank is 67 over $\bC$ or 66 over $\bar\bF_3$.

Now from the $\E_2$ page on, there is nowhere for the differentials to go, so the spectral sequence degenerates, giving $h^*(T_X) = 0, 6, 6, 0$ over $\bC$ or $1, 8, 7, 0$ over $\bar\bF_3$.
\end{proof}

\begin{remark}
It may be surprising that $h^0(T_X) = 1$ when we work over $\bar\bF_3$, meaning that $X$ carries a non-trivial vector field, because this cannot happen for Calabi--Yau threefolds defined over $\bC$.  The vector field comes from the action of $\mu_3$ on $X$ generated by $\tau^2$, which will turn out to be free (Proposition \ref{prop:fixed_and_free}), and yields an injection $\mu_3 \hookrightarrow \Aut(X)$ as we mentioned in Remark \ref{rem:intro}(c).  Over $\bar\bF_3$, the finite group scheme $\mu_3$ is non-reduced, and the tangent space at the identity is 1-dimensional, giving a non-zero element of $\H^0(T_X)$.  Over the complex numbers, $\mu_3$ still acts freely on $X$, but it is a discrete group isomorphic to $\bZ/3$.
\end{remark}

\begin{corollary}\label{cor:hodge numbers of X}
Over $\bar\bF_3$, the Hodge numbers of $X$ are as claimed in Theorem \ref{main_thm}. Over $\bC$, they are as shown in \eqref{eq:1661}.
\end{corollary}
\begin{proof}
By Proposition \ref{sec2_prop}(b), we have $\omega_X \cong \cO_X$, and therefore $\Omega^2_X \cong T_X$.  The claimed values of $h^*(\Omega^2_X)$ follow from Proposition \ref{sec2_prop}(d), and those of $h^*(\Omega_X)$ follow by Serre duality.  The Hodge numbers $h^*(\cO_X)=h^*(\Omega^3_X)$ are given in Proposition \ref{sec2_prop}(c).
\end{proof}


\section{Construction of \texorpdfstring{$M$}{M} and the derived equivalence} \label{sec:constr_of_M}

In the last section we constructed a Calabi--Yau threefold $X \subset \bP^5 \times \bP^1$, defined over either $\bC$ or $\bar\bF_3$.  The projection $\pi_2\colon X \to \bP^1$ has many sections coming from $Z' \times \bP^1 \subset X$, where $Z' \subset Y \subset \bP^5$ is a subscheme of length 72 that appeared in the proof of Proposition \ref{sec2_prop}.  In this section we show that $\pi_2$ is an Abelian surface fibration, we will construct a dual Abelian surface fibration $M \to \bP^1$, depending on a choice of section of $\pi_2$, and we will show that there is a universal sheaf on $X \times_{\bP^1} M$ that induces an equivalence $D^b(X) \cong D^b(M)$ following Schnell's argument in \cite[\S5]{schnell}.

\begin{proposition} \label{prop:ab_fib}
The map $\pi_2\colon X \to \bP^1$ is flat.  A general fiber is smooth, and is an Abelian surface.\end{proposition}
\begin{proof}
Flatness follows from \cite[III Prop.~9.7]{hartshorne}, because $X$ is integral, $\bP^1$ is a smooth curve, and $\pi_2$ has sections and thus is dominant.

With Macaulay2 we find a smooth fiber $A$ of $\pi_2$, continuing to use the ancillary file \verb|verify.m2|.  The normal bundle of $A$ in $X$ is trivial, because it is the fiber of a map to a smooth curve, so by the adjunction formula we have $\omega_A = \cO_A$, and the resolution \eqref{eq:res} gives $\chi(\cO_A) = 0$.  Over $\bC$, this implies that $A$ is an Abelian surface.  Over $\bar\bF_3$, Bombieri and Mumford's classification \cite{bombieri-mumford2} implies that $A$ is either an Abelian surface or a quasi-hyperelliptic surface.  We will show that the former holds by computing the second Betti number of $A$.

Let $k = \bar\bF_3$, let $W$ be the ring of Witt vectors of $k$, and let $K$ be the field of fractions of $W$.  Our point $p \in \bP^3$ defined over $k$ lifts to a point defined over $W$, so our equations for the threefold $Y$ and for the rational map $Y \dashrightarrow \bP^1$ lift to $W$, so the graph $X$ of the rational map lifts as well.  Let $X_W \to \Spec(W)$ be the lift.  The special fiber $X = X_k$ is smooth, and the general fiber $X_K$ is 3-dimensional, so it is smooth as well.

The fibers of $X_k \to \bP^1_k$ are 2-dimensional, so the fibers of $X_K \to \bP^1_K$ are at most 2-dimensional by semicontinuity, and at least 2-dimensional because they are fibers of a map from a threefold to a smooth curve.  Thus every fiber of $X_k \to \bP^1_k$ lifts to characteristic zero, and the smooth fibers lift to smooth surfaces with $\omega = \cO$ and $\chi(\cO) = 0$.  Thus the fibers over $K$ are Abelian surfaces and have $b_2 = 6$, so the fibers over $k$ also have $b_2 = 6$, so they are Abelian surfaces by \cite[Thm.~6]{bombieri-mumford2}.
\end{proof}

We want to construct a dual Abelian surface fibration: precisely, let $V \subset \bP^1$ be the open set over which $\pi_2$ is smooth, let $X^\circ = \pi_2^{-1}(V) \subset X$, and let $M^\circ \to V$ be the relative $\sPic^0$ of $X^\circ \to V$; then we want to construct $M \to \bP^1$ as a compactification of $M^\circ$.

If the fibers of $\pi_2$ were all integral, we could take $M$ to be the closure of $M^\circ$ in the moduli space of semistable sheaves of rank 1 on the fibers of $\pi_2$.  Then all the sheaves would in fact be stable, and because $\pi_2$ has sections, there would be a universal sheaf on $X \times_{\bP^1} M$, which would eventually induce a derived equivalence $D^b(X) \cong D^b(M)$.  But $\pi_2$ has some reducible fibers, and we are unable to avoid properly semistable sheaves, even by varying the polarization.

We can overcome this issue by choosing a preferred section $\Sigma \cong \bP^1 \subset X$ of the projection $\pi_2$.  Really we want to define $M$ as the closure of $M^\circ$ in a moduli space of framed sheaves, where a point of the moduli space parametrizes a pure sheaf $L$ on the fiber $A \subset X$, with the same Hilbert polynomial as $\cO_A$, together with a surjection $L \twoheadrightarrow \cO_{A \cap \Sigma}$.  But the literature on framed sheaves only deals with smooth varieties in characteristic 0; perhaps it is straightforward to adapt the theory to families of varieties, with reducible fibers, in positive characteristic, but we do not want to carry out that foundational work, nor to leave it as an exercise to the reader.  Instead we will blow up the section and take the closure of $M^\circ$ in a moduli space of sheaves on the blow-up, and then show that those sheaves descend to $X$.  Yet another approach would have been to take the moduli space of rank-1 pure sheaves $L$ such that $L \otimes \cI_\Sigma$ is stable, but this is technically more difficult. \bigskip

So we construct $M$.  Fix a section $\Sigma \cong \bP^1 \subset X$ of $\pi_2$, and let $\tilde X$ be the blow-up of $X$ along $\Sigma$ as shown:
\[ \xymatrix@C=2pt{
&&& \tilde X \ar[d]_-\varpi \ar@{}[r]|{\supset} & E \ar[d] \\
\bP^5 &&& X \ar[lll]_-{\pi_1} \ar[d]_-{\pi_2} \ar@{}[r]|{\supset} & \Sigma \\
&&& \bP^1 \ar[ru]_-\sigma.
}\]
\begin{lemma} \label{lem:rel_ample}
Let $E \subset \tilde X$ be the exceptional divisor of the blow-up $\varpi$, and let $H \in \Pic(X)$ be the pullback of the hyperplane class from $\bP^5$, which is relatively ample for $\pi_2$.  Then $\varpi^* 2H - E$ is relatively ample for $\pi_2 \circ \varpi$.
\end{lemma}
\begin{proof}
The fibers of $\pi_2$ are embedded in $\bP^5$, so the fibers of $\pi_2 \circ \varpi$ are embedded in the blow-up of $\bP^5$ at a point.  The latter is naturally embedded in $\bP^5 \times \bP^4$, and the restriction of $\cO_{\bP^5 \times \bP^4}(1,1)$ turns out to be $\varpi^* 2H - E$.
\end{proof}

\begin{proposition}
Let $X^\circ \to V$ and $M^\circ \to V$ be as above.  Embed $M^\circ$ into the moduli space of $(\varpi^* 2H - E)$-semistable sheaves on the fibers of $\pi_2 \circ \varpi$ by sending a line bundle $L$ on a fiber of $\pi_2$ to $\varpi^* L$, and let $M$ be the closure of $M^\circ$ in this embedding.  Then
\begin{enumerate}
\addtolength \itemsep \smallskipamount
\item $M$ is flat over $\bP^1$,
\item $M$ parametrizes only stable sheaves,
\item there is a universal sheaf $\tilde P$ on $\tilde X \times_{\bP^1} M$, and
\item $\tilde P$ descends a sheaf $P$ on $X \times_{\bP^1} M$.
\end{enumerate}
\end{proposition}
\begin{proof}
(a) Flatness follows from \cite[III Prop.~9.8]{hartshorne}, because $M^\circ$ was flat over $V \subset \bP^1$ and $M$ was obtained by taking the closure in a projective scheme over $\bP^1$. \smallskip

(b) (c)  Let $A \subset \bP^5$ be a smooth fiber of $\pi_2$, and let $\tilde A$ be its blow-up at the point $A \cap \Sigma$.  From the resolution \eqref{eq:res} we find that the Hilbert polynomial of $\cO_A$ with respect to $H$ is $6t^2$.  Thus the Hilbert polynomial of $\cO_{\tilde A}$ with respect to $\varpi^* 2H - E$ is
\begin{align*}
\chi(\varpi^* \cO_A(2tH) \otimes \cO_{\tilde A}(-tE))
&= \chi(\cO_A(2tH) \otimes \varpi_* \cO_{\tilde A}(-tE)) \\
&= \chi(\cO_A(2tH) \otimes \cI_{\text{pt}/A}^t).
\end{align*}
In K-theory, we have $\cI_{\text{pt}/A}^t = \cO_A - \tbinom{t+1}{2} \cO_{\text{pt}}$, so this Hilbert polynomial equals
\[ 6(2t)^2 - \tbinom{t+1}{2} = 47\tbinom{t+1}{2}  - 24t. \]
Because 47 and 24 are coprime, all semistable sheaves with this Hilbert polynomial are stable, and there is a universal sheaf $\tilde P$ on $\tilde X \times_{\bP^1} M$ by \cite[Cor.~4.6.6]{huybrechts_lehn}. \smallskip

(d) We set $P = (\varpi \times 1)_* \tilde P$, and we argue that the counit of the adjunction
\[ (\varpi \times 1)^* (\varpi \times 1)_* \tilde P \to \tilde P \]
is an isomorphism.  Because $\tilde P$ is flat over $M$, it is enough to check that the restriction of this map to the fiber of $\tilde X \times_{\bP^1} M$ over any closed point of $M$ is an isomorphism.
Suppose that we are given a closed point of $M$ parametrizing a sheaf $L$ on a fiber $\tilde A \subset \tilde X$ of $\pi_2 \circ \varpi$.  Because $\tilde X$ is flat over $\bP^1$, the restriction of the counit above is just the counit
\[ \varpi^* \varpi_* L \to L, \]
and we want to argue that this is isomorphism.  Because $\tilde A$ is the blow-up of a fiber $A \subset X$ of $\pi_2$ at a smooth point, it is enough to show that $L$ is locally free in a neighborhood of the exceptional line in $\tilde A$, and that its restriction to that exceptional line has degree zero.

For this last claim we appeal to Altman and Kleiman \cite{altman-kleiman2}, the lemma in Step XII of the proof of Theorem 3.2, which states that for a family of rank-1 torsion-free sheaves on the fibers of a smooth, finitely presentable morphism with integral geometric fibers, being locally free is a closed condition.  (It is also an open condition, but that is more elementary.)  Our map $\pi_2\colon X \to \bP^1$ has finitely many singular fibers, of which some are reducible, but we can get an open set $W \subset X$ that contains the section $\Sigma$ by deleting the points at which $\pi_2$ is not smooth, and for the reducible fibers, deleting the components that do not meet $\Sigma$.  Then we apply the lemma to the restriction of $\tilde P$ to $\varpi^{-1}(W) \times_{\bP^1} M$, and recall that $M$ was defined as the closure of a space parametrizing line bundles.  Moreover, because those line bundles were pulled back from $X$, their restrictions to the exceptional $\bP^1$s had degree zero, so this remains true in the limit.
\end{proof}

We will find that the following theorem applies to our $X$, $M$, and $P$.

\begin{theorem}[Bridgeland and Maciocia] \label{bm_thm}
Let $X$ be a smooth projective variety of dimension $n$ over an algebraically closed field $k$ of arbitrary characteristic.  Let $M$ be an irreducible projective scheme of the same dimension $n$ over $k$.  Let $P$ be a sheaf on $X \times M$, flat over $M$.

Suppose that for each closed point $y \in M$, the sheaf $P_y$ on $X$ satisfies $\Hom_X(P_y,P_y) = k$, the Kodaira--Spencer map $T_y Y \to \Ext^1_X(P_y,P_y)$ is injective, and $P_y \otimes \omega_X \cong P_y$.  Suppose that for distinct closed points $y_1, y_2 \in M$ we have $\Hom_X(P_{y_1},P_{y_2}) = 0$, and that the closed subscheme
\[ \Gamma = \{ (y_1,y_2) \in M \times M : \Ext^i_X(P_{y_1}, P_{y_2}) \ne 0 \text{ for some } i \in \bZ \} \]
has dimension at most $n+1$.

Then $M$ is smooth, and the functor $D^b(X) \to D^b(M)$ induced by $P$ is an equivalence.
\end{theorem}
\begin{proof}
This is a slight generalization of \cite[Thm.~6.1]{bridgeland-maciocia}, and the proof is essentially the same; we only remark on the changes needed.  First, while Bridgeland and Maciocia state all their results over $\bC$, they do not use any special properties of working in characteristic zero such as generic smoothness, and the crucial \cite[Thm.~4.3]{bridgeland-maciocia} is valid over an arbitrary field, as Bridgeland and Iyengar show in \cite{bridgeland-iyengar} which corrects the proof of \cite[Thm.~4.3]{bridgeland-maciocia}.  On the other hand, they reason extensively with closed points, so we retain the hypothesis that $k$ is algebraically closed.

The other difference between Bridgeland and Maciocia's statement and ours is that they require the Kodaira--Spencer map to be an isomorphism, whereas we only require it to be injective.  But inspecting the last paragraph of their proof, we see that injectivity is enough.  And in fact Bridgeland and Maciocia use this in their own application \cite[Thm.~1.2]{bridgeland-maciocia}, because they ask for an irreducible component, rather than a connected component, of the moduli space of sheaves.  The latest arXiv version of their paper includes a new Footnote 2 saying the same thing.
\end{proof}

\begin{proposition}
Theorem \ref{bm_thm} applies to the $X$, $M$, and $P$ constructed in this section.  Thus $M$ is smooth, and $D^b(M) \cong D^b(X)$.
\end{proposition}
\begin{proof}
This is entirely similar to \cite[\S7.3]{bridgeland-maciocia}, but for the reader's convenience we go through the details.

Let $y_1, y_2 \in M$.  We have
\[ \Hom_X(P_{y_1}, P_{y_2}) = \Hom_{\tilde X}(\varpi^* P_{y_1}, \varpi^* P_{y_2}) \]
because $\varpi_* \cO_{\tilde X} = \cO_X$.  Because the $\varpi^* P_{y_i}$ are stable sheaves, this Hom space is 1-dimensional if $y_1 = y_2$ and zero if $y_1 \ne y_2$.  Similarly, the Kodaira-Spencer map is injective because $\varpi^*$ is fully faithful and $M$ is a subvariety of the moduli space of stable sheaves on $\tilde X$.  We have $P_{y_i} \otimes \omega_X \cong P_{y_i}$ because $\omega_X \cong \cO_X$.

It remains to show that $\Gamma$ is at most 4-dimensional.  If $y_1, y_2 \in M$ map to different points of $\bP^1$, then $P_{y_1}$ and $P_{y_2}$ are supported on different fibers of $\pi_2\colon X \to \bP^1$, so there are no Exts between them.  If $y_1$ and $y_2$ are distinct but map to the same point of $V \subset \bP^1$, then $P_{y_1}$ and $P_{y_2}$ are different line bundles of degree 0 on a smooth Abelian surface, so again there are no Exts between them.  Thus we see that
\[ \Gamma \subset (N_1 \times N_1) \cup \dotsb \cup (N_k \times N_k) \cup \Delta, \]
where $N_1, \dotsc, N_k \subset M$ are the fibers over the finite set $\bP^1 \setminus V$.  These fibers are 2-dimensional because $M$ is flat over $\bP^1$, and $\Delta$ is 3-dimensional, so $\Gamma$ is at most 4-dimensional, as desired.
\end{proof}


\section{The Picard scheme of \texorpdfstring{$M$}{M}}

Now we have smooth Calabi--Yau threefolds $X$ and $M$, defined over either $\bC$ or $\bar\bF_3$, and dual Abelian surface fibrations $X \to \bP^1$ and $M \to \bP^1$.  In this section we compute $\sPic^\tau_M$, the ``torsion component of the identity'' of the Picard scheme of $M$ which parametrizes numerically trivial line bundles.  In the next section we will use it to find the Hodge numbers of $M$.

\begin{proposition}
Over $\bC$, the topological fundamental group $\piOneTop(X) = 0$.  Over $\bar\bF_3$, the \'etale fundamental group $\piOneEt(X) = 0$.
\end{proposition}
\begin{proof}
Over $\bC$ we follow Gross and Pavanelli \cite[Thm.~1.4]{grpa}.  A smooth complete intersection of two cubics in $\bP^5$ is simply connected by the Lefschetz hyperplane theorem.  The degeneration to a nodal complete intersection $Y \subset \bP^5$, followed by the the small resolution $X \to Y$, has the effect of cutting out a $D^3 \times S^3$ for each node, and gluing in an $S^2 \times D^4$, which leaves $\piOneTop$ unchanged by van Kampen's theorem.  Thus $\piOneTop(X) = 0$, and its profinite completion $\piOneEt(X) = 0$ as well.

Over $\bar\bF_3$, we lift $X$ to $W$ as in the proof of Proposition \ref{prop:ab_fib}.  The geometric general fiber has $\piOneEt = 0$, and the specialization map on $\piOneEt$ is surjective by \cite[\href{https://stacks.math.columbia.edu/tag/0C0P}{Tag 0C0P}]{stacks-project}, so the special fiber has $\piOneEt = 0$.
\end{proof}

\begin{proposition} \label{prop:Z6}
$\sPic_X = \bZ^6$.
\end{proposition}
\begin{proof}
Over $\bC$, we know that the Picard group is free Abelian because $\piOneTop(X) = 0$, and we know its rank by the Lefschetz theorem of $(1,1)$ classes and the Hodge numbers given in \eqref{eq:1661}. We conclude that the group scheme $\sPic_X$ is isomorphic to $\bZ^6$.

Over $\bar\bF_3$, we again lift $X$ to $W$, and apply \cite[Prop.~4.2]{ehsb} to deduce that $\sPic$ is flat over $\Spec(W)$; to use this reference, recall that $h^1(\cO_X) = h^2(\cO_X) = 0$ by Proposition \ref{sec2_prop}(c).
\end{proof}

\begin{proposition} \label{prop:pictau} \ 
\begin{enumerate}
\addtolength \itemsep \smallskipamount
\item $\sPic^\tau_M = \bZ/3 \times \mu_3$.
\item Over $\bC$ we have $\piOneTop(M) = (\bZ/3)^2$. Over $\bar\bF_3$ we have $\piOneEt(M) = \bZ/3$.
\end{enumerate}
\end{proposition}

\noindent Let us outline the proof, which will occupy the rest of the section.  At the beginning of our construction in \S\ref{sec:construction}, the group $G := \bZ/6 \times \mu_6$ acted on $\bP^5$ by powers of $\sigma$ and $\tau$.  We will show that $G$ acts on $X$, though not freely, and that $X/G$ is birational to $M$.  We will take a crepant resolution $\hat X/G$ of $X/G$, giving a smooth threefold whose canonical bundle is numerically trival.  Thus by \cite[Cor.~3.54]{kollar-mori}, the birational map $\hat X/G \dashrightarrow M$ is an isomorphism away from a set of codimension 2 on either side,\footnote{In fact we expect that the birational map is an isomorphism, but to prove this, following Schnell's argument in \cite[Lem.~5.4]{schnell}, would require a detailed analysis of the singular fibers of $X \to \bP^1$ and $M \to \bP^1$.} so the Picard schemes of $\hat X/G$ and $M$ are isomorphic.  We will conclude by computing $\sPic^\tau_{\hat X/G}$.

\begin{lemma}
Let $G = \bZ/6 \times \mu_6$ act on $\bP^5$ by powers of $\sigma$ and $\tau$ as introduced at the beginning of \S\ref{sec:construction}.  Then $G$ acts on the singular threefold $Y \subset \bP^5$ and its small resolution $X \subset \bP^5 \times \bP^1$.
\end{lemma}
\begin{proof}
The cubic polynomials $f_0, \dotsc, f_3$ were preserved by $\sigma^2$ and $\tau^2$ by construction, and in fact they are preserved by $\tau$ as well.  Thus $f_p$, which is a linear combination of $f_0,\dotsc,f_3$, is preserved by $\sigma^2$ and $\tau$.  And $Y$, which is cut out by $f_p$ and $\sigma f_p$, is preserved by all of $G$.  The rational map $\phi\colon Y \dashrightarrow \bP^1 \subset \Gr(2,4)$ given by the $2 \times 2$ minors of the matrix \eqref{matrix_formerly_known_as_M} is $G$-equivariant, with $G$ acting trivially on the target, so $G$ perserves the graph of $\phi$, which is $X$.
\end{proof}

\begin{proposition}
The quotient $X/G$ is birational to $M$.
\end{proposition}
\begin{proof}
As in the last section, let $V \subset \bP^1$ be the open set over which $X$ is smooth, and consider the smooth fibrations $X^\circ \to V$ and $M^\circ \to V$.  We have fixed a section of $X^\circ \to V$, so these are dual Abelian schemes, and $G$ becomes a subgroup scheme of $X^\circ$ by acting on the section.  The polarization $H$ from Lemma \ref{lem:rel_ample} determines an isogeny $X^\circ \to M^\circ$, and because $G$ preserves $H$, we see that $G$ is contained in the kernel of this isogeny.  We have seen that $\chi(\cO_A(H)) = 6$ for a fiber $A \subset X^\circ$, so the kernel has degree $6^2 = 36$ by \cite[\S III.16]{mumford_ab_var}, so $G$ is the whole kernel, and we conclude that $X^\circ/G$ is isomorphic to $M^\circ$.
\end{proof}

\begin{proposition} \label{prop:fixed_and_free}
Write $G = G' \times G''$, where $G' = \bZ_2 \times \mu_2$ is generated by $\sigma^3$ and $\tau^3$, and $G'' = \bZ/3 \times \mu_3$ is generated by $\sigma^2$ and $\tau^2$.  Acting on $X$, each of the three non-trivial element of $G'$ fixes a pair of disjoint, smooth, genus-1 curves, and all six curves are disjoint from one another.  The action of $G''$ on $X/G'$ is free.
\end{proposition}
\begin{proof}
We check this using Macaulay2, continuing to use the ancillary file \verb|verify.m2|.
\end{proof}

Let $C_1, C_2 \subset X$ be the curves fixed by $\sigma^3$, let $C_3$ and $C_4$ be the curves fixed by $\tau^3$, and let $C_5$ and $C_6$ be the curves fixed by $\sigma^3 \tau^3$.  Observe that $\sigma^3$ switches $C_3 \cup C_4$ with $C_5 \cup C_6$.

Let $\hat X \to X$ be the blow-up of $X$ along these six curves, with exceptional divisors $E_1, \dotsc, E_6 \subset \hat X$.  Then $\hat X/G'$ is smooth, by a special case of the Chevalley--Shephard--Todd theorem: the involution $\sigma^3$ acts on $\hat X$ fixing a disjoint union of smooth divisors, so $\hat X/\sigma^3$ is smooth, and then the same can be said of $\tau^3$ acting on $\hat X/\sigma^3$.

\begin{proposition}
The canonical bundle of $\hat X/G'$ is trivial.
\end{proposition}
\begin{proof}
We know that $\omega_X = \cO_X$, so $\omega_{\hat X} = \cO_{\hat X}(E_1 + \dotsb + E_6)$.  To pass from this to $\omega_{\hat X/G'}$ is a Riemann--Hurwitz type calculation.  The branched double cover $f\colon \hat X \to \hat X/\sigma^3$ is ramified along $E_1 \cup E_2$, so the relative canonical bundle  of $f$ is $\cO_{\hat X}(E_1 + E_2)$, so the canonical bundle of the quotient is $\cO(f(E_3) + f(E_4))$.  Similarly, the branched double cover $\hat X/\sigma^3 \to \hat X/G'$ is ramified along $f(E_3) \cup f(E_4)$, so the canonical bundle of $\hat X/G'$ is trivial.
\end{proof}

\begin{proposition} \label{prop:pi_one_X_hat_mod_G}
Over $\bC$ we have $\piOneTop(\hat X/G') = 0$.  Over $\bar\bF_3$ we have $\piOneEt(\hat X/G') = 0$.
Over either field we have $\sPic^\tau_{\hat X/G'} = 0$.
\end{proposition}
\begin{proof}
Because $\hat X$ is the blow-up of $X$ along six disjoint smooth curves, we have $\pi_1(\hat X) = 0$ over either field,
and $\sPic_{\hat X} = \bZ^{12}$.  We will argue that the branched double cover
\[ f\colon \hat X \to \hat X/\sigma^3 \]
induces a surjection on $\pi_1$ and an injection on Picard schemes.  Then the same argument applies to the branched double cover $\hat X/\sigma^3 \to \hat X/G'$, and the proposition follows.

For surjectivity on $\piOneEt$ over $\bar\bF_3$, we use a result of K\'ollar \cite[Cor.~6]{kollar_pi_1}: $f$ is flat, hence is universally open, and the fiber product $\hat X \times_{\hat X/G'} \hat X$ is connected because the diagonal and the anti-diagonal meet along the ramification divisor.  Koll\'ar's result also holds for $\piOneTop$ over $\bC$, as he remarks in \cite[Paragraph 11]{kollar_pi_1}.

For injectivity on Picard schemes, we first claim that the kernel of
\[ f^*\colon \sPic_{\hat X/\sigma^3} \to \sPic_{\hat X} \]
is 2-torsion.  If $L \in \Pic(\hat X/\sigma^3)$ satisfies $f^* L = \cO_{\hat X}$, then $f_* f^* L = f_* \cO_{\hat X}$, so $L \otimes f_* \cO_{\hat X} = f_* \cO_{\hat X}$.  Taking determinants, and noting that $f_* \cO_{\hat X}$ is a vector bundle of rank 2, we see that $L^{\otimes 2}$ is trivial.  This argument works in families, so it works for Picard schemes as well as Picard groups.

Now a 2-torsion group scheme over an algebraically closed field of characteristic $\ne 2$ is discrete, so it is enough to show that there is no 2-torsion in the Picard \emph{group} of $\hat X/\sigma^3$.  But a line bundle $L \in \Pic(\hat X/\sigma^3)$ with $L^{\otimes 2} = \cO$ would give an \'etale double cover, which is impossible because $\pi_1(\hat X/\sigma^3) = 0$.
\end{proof}

\begin{proposition} \ 
\begin{enumerate}
\item The quotient $\hat X/G$ is smooth, its canonical bundle is numerically trivial, and $\sPic^\tau_{\hat X/G} = \bZ/3 \times \mu_3$.
\item Over $\bC$ we have $\piOneTop(\hat X/G) = (\bZ/3)^2$. \\
Over $\bar\bF_3$ we have $\piOneEt(\hat X/G)=\bZ/3$.
\end{enumerate}
\end{proposition}
\begin{proof}
(a) The quotient $\hat X/G = (\hat X/G')/G''$ is smooth because $\hat X/G'$ is smooth and $G''$ acts freely on $\hat X/G'$.  We know that $\omega_{\hat X/G'}$ is trivial, so $\omega_{\hat X/G}$ is torsion, hence is numerically trivial.
We know that $\sPic^\tau_{\hat X/G'} = 0$, so $\sPic^\tau_{\hat X/G}$ is isomorphic to the Cartier dual of $G''$ by \cite[Cor.~1.2]{remy}, which draws on \cite{jensen}.  Note that $G'' = \bZ/3 \times \mu_3$ is its own Cartier dual. \smallskip

(b) Over $\bC$ we have $\piOneTop(\hat X/G') = 0$ by Proposition \ref{prop:pi_one_X_hat_mod_G}.  The quotient map $\hat X/G' \to \hat X/G$ is \'etale and gives $\piOneTop(\hat X/G) \cong G'' \cong (\bZ/3)^2$.

Over $\bar\bF_3$ we have $\piOneEt(\hat X/G') = 0$, again by Proposition \ref{prop:pi_one_X_hat_mod_G}.  We factor the quotient map $\hat X/G' \to \hat X/G$ into a quotient by $\mu_3$, which is a universal homeomorphism and hence induces an isomorphism on $\piOneEt$ \cite[\href{https://stacks.math.columbia.edu/tag/0BQN}{Tag 0BQN}]{stacks-project}, followed by a quotient by $\bZ/3$, which is \'etale and gives $\piOneEt(\hat X/G) \cong \bZ/3$.
\end{proof}


\section{Hodge numbers of \texorpdfstring{$M$}{M}} \label{sec:finish_the_proof}

In \S\ref{sec:construction} we found the Hodge numbers of $X$, and in the last section we found $\sPic^\tau_M$.  In this section we deduce the Hodge numbers of $M$, completing the proof of Theorem \ref{main_thm}.  We also record the Hochschild homology groups, de Rham cohomology groups, and Betti numbers, for use in the next section.

\begin{proposition}
The Hochschild homology groups of $X$ and $M$ are isomorphic.  Over $\bC$, their dimensions are:
\[ \def\arraystretch{1.2} %
\begin{array}{c|*{7}{C{3ex}}}
i & $-3$ & $-2$ & $-1$ & $0$ & $1$ & $2$ & $3$ \\ \hline
hh_i & $1$ & $0$ & $6$ & $14$ & $6$ & $0$ & $1$
\end{array} \]
Over $\bar\bF_3$, their dimensions are:
\[ \def\arraystretch{1.2} %
\begin{array}{c|*{7}{C{3ex}}}
i & $-3$ & $-2$ & $-1$ & $0$ & $1$ & $2$ & $3$ \\ \hline
hh_i & $1$ & $1$ & $8$ & $16$ & $8$ & $1$ & $1$
\end{array} \]
\end{proposition}
\begin{proof}
Hochschild homology is an invariant of the derived category, and we have seen that $D^b(X) \cong D^b(M)$.  Over $\bC$, the Hochschild--Kostant--Rosenberg theorem implies that the Hochschild numbers are the sums of the columns of the Hodge diamond.  Over $\bar\bF_3$, a result Antieau and Vezzosi \cite[Thm.~1.3]{antieau-vezzosi} gives the same conclusion, because $X$ is a smooth projective threefold.  The Hodge numbers of $X$ were determined in Corollary \ref{cor:hodge numbers of X}.
\end{proof}

\begin{proposition}
Over $\bC$, the Hodge numbers of $M$ are equal to those of $X$.  Over $\bar\bF_3$, they are as claimed in Theorem \ref{main_thm}.
\end{proposition}
\begin{proof}
Over $\bC$ this is automatic, as discussed in the introduction, so we focus on $\bar\bF_3$.

For general reasons, $\H^1(\cO_M)$ is the tangent space to $\sPic_M$ at the origin.  By Proposition \ref{prop:pictau}(a) we have $\sPic^\tau_M = \bZ/3 \times \mu_3$, so $\sPic^0_M=\mu_3$, so $h^{0,1} = 1$.  Because $\omega_M \cong \cO_M$ we get $h^{3,1} = 1$, and by Serre duality we get $h^{3,2} = h^{0,2} = 1$.  

Again by \cite[Thm.~1.3]{antieau-vezzosi}, the Hochschild numbers of $M$ are the sums of the columns of the Hodge diamond, so the Hodge diamond must be of the form
\[ 
\arraycolsep=4pt
\begin{array}{ccccccccccccccc}
  &   &   & 1 &   &   &   \\
  &   & 1 &   & x &   &   \\
  & 1 &   & 7 &   & 0 &   \\
1 &   & y &   & y &   & 1 \\
  & 0 &   & 7 &   & 1 &   \\
  &   & x &   & 1 &   &   \\
  &   &   & 1\makebox[0pt]{ ,}&   &   &   
\end{array}
\]
where $x+y = 7$.  We will argue that $x = h^{1,0}$ is equal to 1.

A theorem of Oda \cite[Cor.~5.12]{oda} gives an injection
\begin{equation}\label{eq:Odas map}
\DM(\sPic_M[V])\hookrightarrow\H^0(\Omega^1_M),
\end{equation}
where $\sPic_M[V]$ is the subgroup scheme of $\sPic_M$ annihilated by the Verschiebung operator and $\DM(\sPic_M[V])$ is its covariant Dieudonn\'e module.   Because $VF = FV = 3$, we see that $\sPic_M[V]$ is contained in the 3-torsion subgroup $\sPic_M[3] = \bZ/3 \times \mu_3$.  Frobenius acts as the identity on $\bZ/3$ and annihilates $\mu_3$, and Verschiebung does the reverse, so $\sPic_M[V] = \bZ/3\bZ$, whose Dieudonn\'e module is $\bar\bF_3$.

Oda's result also describes the image of~\eqref{eq:Odas map} as the subspace of indefinitely closed 1-forms.  A 1-form $\alpha\in\H^0(\Omega^1_M)$ is called \emph{indefinitely closed} if $d\alpha = 0$, $d(C \alpha) = 0$, $d(C^2 \alpha) = 0$, and so on, where $d$ is the de Rham differential and $C^n$ is the iterated Cartier operator, defined on $\ker(d \circ C^{n-1}) \subset \H^0(\Omega^1_M)$.\footnote{Oda denotes the Cartier operator by $V$ rather than $C$; see \cite[Def.~5.5]{oda}.}  Our case is particularly simple: we have seen that $\H^0(\Omega^2_M)=0$, so $d=0$, so every 1-form is indefinitely closed, and hence \eqref{eq:Odas map} is an isomorphism.
\end{proof}

\begin{proposition} \label{prop:deRham}
Over $\bC$, the de Rham cohomology groups of $X$ and $M$ are isomorphic, and their dimensions are:
\[ \def\arraystretch{1.2}
\begin{array}{c|*{7}{C{3ex}}}
i & $0$ & $1$ & $2$ & $3$ & $4$ & $5$ & $6$ \\ \hline
h^i_{\dR} & $1$ & $0$ & $6$ & $14$ & $6$ & $0$ & $1$
\end{array} \]
Over $\bar\bF_3$, the de Rham cohomology groups of $X$ and $M$ are different, and their dimensions are:
\[ \def\arraystretch{1.2}
\begin{array}{c|*{7}{C{3ex}}}
i & $0$ & $1$ & $2$ & $3$ & $4$ & $5$ & $6$ \\ \hline
h^i_{\dR}(X/\bar\bF_3) & $1$ & $0$ & $8$ & $18$ & $8$ & $0$ & $1$ \\
h^i_{\dR}(M/\bar\bF_3) & $1$ & $2$ & $8$ & $14$ & $8$ & $2$ & $1$
\end{array} \]
\end{proposition}
\begin{proof}
Over $\bC$, the Hodge--de Rham spectral sequence always degenerates at the $\E_1$ page, so the de Rham numbers are the sums of the rows of the Hodge diamond.  Over $\bar\bF_3$, the spectral sequence degenerates for $X$ because $X$ lifts to $W$, as we saw in the proof of Proposition \ref{prop:ab_fib}.  Thus it degenerates for $M$ by \cite[Thm.~2.6]{antieau_bragg}.
\end{proof}

\begin{proposition} \label{prop:betti}
Over either $\bC$ or $\bar\bF_3$, the Betti numbers of $X$ and $M$ are:
\[ \def\arraystretch{1.2}
\begin{array}{c|*{7}{C{3ex}}}
i & $0$ & $1$ & $2$ & $3$ & $4$ & $5$ & $6$ \\ \hline
b_i & $1$ & $0$ & $6$ & $14$ & $6$ & $0$ & $1$
\end{array} \]
\end{proposition}
\begin{proof}
In characteristic zero, the Betti numbers are the same as the de Rham numbers.  Over $\bar\bF_3$, we again use the fact that $X$ lifts to $W$, together with the fact that Betti numbers are deformation invariant, so the Betti numbers of $X$ over $\bar\bF_3$ are the same as in characteristic zero.  Betti numbers are derived invariants of threefolds by \cite[Thm.~1.2(5)]{antieau_bragg}, giving the result for $M$.
\end{proof}

\begin{remark}
Since $X$ lifts to the ring of Witt vectors $W$ of $k := \bar\bF_3$, it is interesting to ask whether $M$ lifts to $W$, or at least to the truncation $W_2$ \cite[Question 2.7]{antieau_bragg}.  If $M$ is in fact isomorphic to $X/G$, and not just birational, then it does lift, at least for our choice of parameters $p \in \bP^3$.  (For other choices there might be an issue with the fixed locus of $G$ looking different over the field of fractions $K$ than it looks over $k$.)  The construction of $M$ as a moduli space need not lift to $W$, because we found a section of $X \to \bP^1$ by choosing a point of $Z \subset \bP^5$, and $Z$ may not have $K$-rational points.  But $Z$ does have rational points over a finite extension of $K$, so at the very least $M$ lifts to a ramified cover of $W$.

On the other hand, we will see in the next section that $M$ is weakly ordinary and hence $F$-split, so it lifts to $W_2$ as Achinger and Zdanowicz discuss in \cite[\S1.3]{achinger-zdanowicz}.
\end{remark}


\section{Hodge\texorpdfstring{--}{-}Witt and crystalline cohomology} \label{sec:crystalline}

In this section we work exclusively over $k := \bar\bF_3$.  Recall that a proper variety of dimension $n$ over a perfect field of positive characteristic is called \emph{weakly ordinary} or \emph{1-ordinary} if the action of absolute Frobenius on $\H^n(\cO)$, also known as the Hasse--Witt matrix, is bijective.

\begin{proposition}
$X$ is weakly ordinary.\footnote{We remind the reader that we have chosen a particular threefold $X$ in a family of threefolds parametrized by $\bP^3$.  We do not claim that \emph{every} smooth member of this family is weakly ordinary, only the one we have chosen, and thus a Zariski open set of them.}
\end{proposition}
\begin{proof}
We have seen that the resolution of singularities $\pi_1\colon X \to Y \subset \bP^5$ satsifies $R\pi_{1*} \cO_X = \cO_Y$, so it is enough to show that $Y$ is weakly ordinary.  Because $Y$ is a complete intersection cut out by two cubic polynomials $f_p$ and $\sigma f_p$ in $k[x_0, \dotsc, x_5]$, the Hasse--Witt matrix is given by the coefficient of $(x_0 x_1 x_2 x_3 x_4 x_5)^2$ in $(f_p \cdot \sigma f_p)^2$: this is similar to well-known fact about elliptic curves \cite[Prop.~IV.4.21]{hartshorne}, whose generalization to higher-dimensional hypersurfaces appears in \cite[Special Case 2.3.7.17]{katz}, and to complete intersections in \cite[Prop.~4.1]{kudo}.  We check that this coefficient is non-zero with Macaulay2.
\end{proof}

We continue to let $W$ be the ring of Witt vectors of $k$, and to let $K = W[1/3]$ be the field of fractions of $W$, which is the maximal unramified extension of $\bQ_p$. 
We let $\sigma$ denote the Frobenius automorphism of $k$ or $W$ or $K$.  This conflicts with our earlier use of $\sigma$ as the generator of $\bZ/6$ acting on $\bP^5$, but it should not cause confusion.

We are interested in the Hodge--Witt cohomology groups $\H^j(W\cO_X)$ and $\H^j(W\Omega^i_X)$, for which our main reference is Illusie's paper \cite{Ill79}, but we also recommend Chambert-Loir's survey paper \cite{chambert-loir}.  They are \emph{Dieudonn\'e modules}, meaning that they are $W$-modules equipped with a $\sigma$-semilinear endomorphism $F$ (Frobenius) and a $\sigma^{-1}$-semilinear endomorphism $V$ (Verschiebung) that satisfy $FV = VF = 3$.

\begin{proposition}\label{prop:weakly_ordinary_derived_invariant}
Weak ordinarity is invariant under derived equivalences.  Thus $M$ is also weakly ordinary.
\end{proposition}
\begin{proof}
For simplicity we deal with threefolds, but the same proof works in higher dimensions.  Start with the exact sequence
\[ 0 \to W\cO_X \xrightarrow{V} W\cO_X \to \cO_X \to 0. \]
Take top cohomology and add Frobenius maps to get the following diagram:
\[ \xymatrix{
\H^3(W\cO_X) \ar[r]^V \ar[d]_F & \H^3(W\cO_X) \ar[r] \ar[d]_F & \H^3(\cO_X) \ar[r] \ar[d]_F & 0 \\
\H^3(W\cO_X) \ar[r]^V & \H^3(W\cO_X) \ar[r] & \H^3(\cO_X) \ar[r] & 0\makebox[0pt]{ .}
} \]
Weak ordinarity means that that the right-hand vertical map is bijective.  This is determined by the left-hand square, which is determined by $\H^3(W\cO_X)$ as a Dieudonn\'e module, and this is a derived invariant by \cite[\S5]{antieau_bragg}.
\end{proof}

We examine the rational crystalline cohomology groups
\[ \H^*(X/K) := \H^*(X/W) \otimes K \]
and their \emph{slopes}, which encode the action of Frobenius. For varieties of dimension $\leq 3$, these are derived invariants by \cite[Thm.~5.15]{antieau_bragg}.

\begin{proposition}
The slopes of $\H^i(X/K)$ and $\H^i(M/K)$ for $i\neq 3$ are as follows:
\[ \def\arraystretch{1.2}
\begin{array}{c|*{2}{C{6ex}}*{3}{C{8ex}}*{2}{C{7ex}}}
i & $0$ & $1$ & $2$ & $3$ & $4$ & $5$ & $6$ \\ \hline
\H^i(-/K) & $K(0)$ & $0$ & $K^6(-1)$ & $?$ & $K^6(-2)$ & $0$ & $K(-3)$
\end{array} \]
\end{proposition}
\begin{proof}
The dimensions are the same as the Betti numbers, which we computed in Proposition \ref{prop:betti}, so it remains to determine the slopes.  For $i=0$ and $i=1$ there is nothing to say.  For $i=2$, we found in Proposition \ref{prop:Z6} that $\Pic(X) = \bZ^6$; the first Chern class map
\[ c_1\colon \Pic(X) \otimes K \to\H^2(X/K)\]
is injective and takes values in the slope-1 part,
so $\H^2(X/K)$ is all slope 1 as claimed.  The claims for $i=4,5,6$ follow by Poincar\'e duality.
\end{proof}

It is difficult to get complete information about the slopes of $\H^3(X/K)\cong\H^3(M/K)\cong K^{14}$.  In theory we could find them by counting points: $X$ is defined over $\bF_9$, and if we could compute $\#X(\bF_{9^k})$ for sufficiently many $k$ then we could determine the zeta function of $X$, which determines the slopes.  But in practice it is infeasible to compute these point counts.

We can get partial information about these slopes, however, from the fact that $X$ and $M$ are weakly ordinary.  We first find the Hodge polygon of $\H^3(M/W)$ using \cite[II Thm.~2.3(ii)]{chambert-loir}.  From the universal coefficient theorem
\begin{equation} \label{eq:univ_coeff}
0 \to \H^3(M/W) \otimes k \to \H^3_{\dR}(M/k) \to \Tor_1^W(\H^4(M/W),k) \to 0
\end{equation}
and the fact that $h^3_{\dR}(M) = b_3(M)$, we see that $\H^3(M/W)$ is torsion free.  The Hodge--de Rham spectral sequence degenerates for $M$, as we saw in the proof of Proposition \ref{prop:deRham}.  Thus the Hodge polygon of $\H^3(M/W)$ is determined by the middle row of the Hodge diamond, which is $1, 6, 6, 1$. We conclude that the Hodge polygon of $\H^3(M/W)$ has vertices 
\[
(0,0),(1,0),(7,6),(13,18),(14,21),
\]
as depicted by the thick line in the following picture.
\bigskip
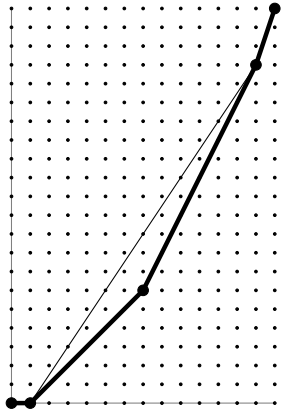
\begin{center}
	\begin{tikzpicture}
	\draw[gray] (0,0) -- (0,21/4);
	\draw[gray] (0,0) -- (14/4,0);
	\draw[line width=.06cm] (0,0) -- (1/4,0);
	\draw[line width=.06cm] (1/4,0) -- (7/4,6/4);
	\draw[line width=.06cm] (7/4,6/4) -- (13/4,18/4);
	\draw[line width=.06cm] (13/4,18/4) -- (14/4,21/4);
	\draw[thick] (0,0) -- (1/4,0);
	\draw[thick] (1/4,0) -- (7/4,6/4);
	\draw[thick] (7/4,6/4) -- (13/4,18/4);
	\draw[thick] (13/4,18/4) -- (14/4,21/4);
	\foreach \x in {0,...,14}
	\foreach \y in {0,...,21} 
	\filldraw[black] (\x/4,\y/4) circle (.5pt);
	\draw (1/4,0) -- (13/4,18/4);
	\filldraw[black] (0,0) circle (2pt);
	\filldraw[black] (1/4,0) circle (2pt);
	\filldraw[black] (7/4,6/4) circle (2pt);
	\filldraw[black] (13/4,18/4) circle (2pt);
	\filldraw[black] (14/4,21/4) circle (2pt);
	\end{tikzpicture}
\end{center}
\medskip
Our calculations with $\TR$ below will show that the Hodge polygons of $\H^3(M/W)$ and $\H^3(X/W)/\text{tors}$ are equal.
The Newton polygon, built from the slopes of $\H^3(M/K)$, is concave up, and lies on or above the Hodge polygon.  Furthermore, by \cite[Prop.~2.4.1]{achinger-zdanowicz}, weak ordinarity of $M$ is equivalent to the first segment of the Newton polygon being the same as that of the Hodge polygon.  Thus the slope 0 part of $\H^3(X/K) \cong \H^3(M/K)$ is 1-dimensional, and by Poincar\'e duality, the slope 3 part is 1-dimensional as well.  We conclude that the Newton polygon lies on or below the polygon with vertices $(0,0),(1,0),(13,18),(14,21)$, depicted as the thin line in the picture above.  Even with the added constraints coming from Poincar\'e duality, there remain many possibilities for the Newton polygon.

Let
	\[
	a=\dim_K\H^3(X/K)_{[1,2)}\hspace{1cm}\mbox{and}\hspace{1cm}b=\dim_K\H^3(X/K)_{[2,3)},
	\]
which satisfy $a+b = 12$.  It seems reasonable to expect that $a=b=6$, so the Hodge and Newton polygons of $\H^3(M/W)$ coincide and $M$ is \emph{strongly ordinary}; this is an open condition in moduli space, but we do not know if it is non-empty. \bigskip

Because $X$ is weakly ordinary and has trivial canonical bundle, it is $F$-split \cite[Prop.~9]{mehta_ramanathan}, and because $\dim X = 3$, this implies that the Hodge--Witt cohomology groups $\H^j(W \Omega^i_X)$ are finitely generated $W$-modules \cite[Cor.~6.2]{joshi}.  Thus by \cite[(2.3)]{illusie_raynaud}, the slope spectral sequence
\[ \E_1^{i,j} = \H^j(W \Omega^i_X) \Rightarrow \H^{i+j}(X/W) \]
degenerates at the $\E_1$ page, and the filtration splits, giving
\[ \H^k(X/W) = \bigoplus_{i+j = k} \H^j(W \Omega^i_X). \]
The same holds for $M$, since it is also weakly ordinary and has trivial canonical bundle.

\begin{proposition} \label{prop:HWofM}
The Hodge--Witt cohomology of $M$ is as follows:
	\[
	\def\arraystretch{1.5}
	\begin{array}
	{c|cccc}
	\H^3 & W      & 0              & k           & W \\
	\H^2 & k      & W^a            & W^6         & k \\
	\H^1 & 0      & W^6 \oplus k   & W^b         & 0\\
	\H^0 & W      & 0              & 0           & W \\ \hline
	     & W\cO_M & W\Omega^1_M    & W\Omega^2_M & W\Omega^3_M
	\end{array}
	\]
All torsion is semisimple\footnote{A finite-length Dieudonn\'e module is called \emph{semisimple torsion} if the action of $F$ is bijective, and \emph{nilpotent torsion} if the action of $F$ is nilpotent.  A general finite-length Dieudonn\'e module can be written uniquely as an extension of semisimple torsion by nilpotent torsion.  
For us the nilpotent torsion will always vanish.} and is annihilated by $V$.
\end{proposition}
\begin{proof}
By \cite[II Cor.~3.5]{Ill79}, we have $\H^j(W \Omega^i_M) \otimes K = \H^{i+j}(M/K)_{[i,i+1)}$, which determines the ranks of the free parts, so we need only determine the torsion parts.

There is no torsion in the bottom row by \cite[II Cor.~2.17]{Ill79}, nor in $\H^1(W\cO_M)$ by \cite[II Prop.~2.19]{Ill79}, nor in $\H^3(W \Omega^3_M)$ by \cite[II Cor.~3.15]{Ill79}. 

There is no torsion along the diagonals $i+j=3$ or $i+j=4$: we have seen that there is no torsion in $\H^3(M/W)$ thanks to the exact sequence \eqref{eq:univ_coeff} and the fact that $h^3_{\dR}(M) = b_3(M)$, and the third term in the sequence shows that there is no torsion in $\H^4(M/W)$ either.

From \cite[II Prop.~6.8]{Ill79} and the fact that $\H^0(W\Omega^2_M) = 0$, we see that the torsion $\H^1(W\Omega^1_M)$ is isomorphic to the torsion in $\NS(M)$ tensored with $W$.  From Proposition \ref{prop:pictau}(a) we know that $\sPic^\tau_M = \bZ/3 \times \mu_3$, so the torsion in $\NS(M)$ is $\bZ/3$, so the torsion in $\H^1(W\Omega^1_M)$ is $k$.  It is semisimple because the action of $F$ on $\bZ/3 \otimes W$ is defined to be $1 \otimes \sigma$.

By \cite[II Rmk.~6.4]{Ill79}, the $V$-torsion submodule $\H^2(W\cO_M)[V^\infty]$ is canonically identified with the covariant Dieudonn\'e module $\DM(\mu_3)$, which is $k$ with $F=\sigma$ and $V=0$.  To see that $\H^2(W\cO_M)$ is entirely $V$-torsion, we follow the proof of \cite[II Prop.~7.3.2]{Ill79}; because $\H^2(W\cO_M) = 0$ and $\H^3(W\cO_M)$ is torsion-free, we get the exact sequence \cite[eq.~(7.3.2.1)]{Ill79}, and the rest of the argument goes through unchanged.

Finally, we get the torsion in $\H^3(W\Omega^2_M)$ and $\H^2(W\Omega^3_M)$ using Ekedahl's duality theorem, summarized in \cite[Thm.~4.4.4(a)]{illusie_finiteness}: the semisimple torsion in $\H^j(W\Omega^i_M)$ is dual to that in $\H^{4-j}(W\Omega^{3-i}_M)$, and the nilpotent torsion in $\H^j(W\Omega^i_M)$ is dual to that in $\H^{5-j}(W\Omega^{2-i}_M)$.
\end{proof}

To transfer information from the Hodge--Witt cohomology of $M$ to that of $X$, we use Hesselholt's $\TR$ invariants, referring to \cite[\S3]{antieau_bragg} for discussion and further references.

\begin{proposition}
The groups $\TR_i(X) = \TR_i(M)$ are as follows:
\[ \def\arraystretch{1.2}
\begin{array}{c|*{2}{C{6ex}}*{3}{C{7ex}}*{2}{C{6ex}}}
i & $-3$ & $-2$ & $-1$ & $0$ & $1$ & $2$ & $3$ \\ \hline
\TR_i & $W$ & $k$ & $W^a \oplus k$ & $?$ & $W^b \oplus k$ & $0$ & $W$,
\end{array} \]
where the $?$ is either $W^{14}$ or $W^{14} \oplus k$.
\end{proposition} 
\begin{proof}
We use the descent spectral sequence
\[ \E_2^{i,j} = \H^j(W\Omega^i_M) \Rightarrow \TR_{i-j}(M) \]
discussed in \cite[Def.~3.5(a)]{antieau_bragg}.  The $\E_2$ page is the same as the table in Proposition \ref{prop:HWofM}, but the differentials go up-up-right. 
We claim that it degenerates at the $\E_2$ page.

The differentials are torsion by \cite[Prop.~3.8]{antieau_bragg}, so the only interesting differential on the $\E_2$ page is
\[ d_2\colon \H^1(W\Omega^1_M) \to \H^3(W\Omega^2_M). \]
The natural maps $\H^j(W\Omega^i_M) \to \H^j(\Omega^i_M)$ give a homomorphism from the descent spectral sequence to the HKR spectral sequence.\footnote{This follows from the compatibility of Hesselholt's de Rham--Witt HKR isomorphism \cite[Thm.~C]{hesselholt} and the usual HKR isomorphism.  The essential verification is the compatibility under linearization of the operator $\delta$ \cite[Def.~1.4.3]{hesselholt} and Connes' $B$ operator, which is \cite[Prop.~1.4.6]{hesselholt}.} Thus we have a commutative diagram
\[ \xymatrix{
\H^1(W\Omega^1_M) \ar[r]^{d_2} \ar[d] & \H^3(W\Omega^2_M) \ar[d] \\
H^1(\Omega^1_M) \ar[r]^{\bar d_2} & \H^3(\Omega^2_M)
} \]
where $\bar d_2$ is the corresponding differential in the HKR spectral sequence.  By \cite[Thm.~1.3]{antieau-vezzosi}, the latter spectral sequence degenerates at $\E_2$, so $\bar d_2=0$.  We have seen that $\H^3(W\Omega^2_M)\cong k$ with $V=0$, which implies that the right vertical arrow is an isomorphism.  We conclude that $d_2=0$.  On the $\E_3$ page, the only possibly nonzero differential
\[ d_3\colon \H^0(W \cO_M) \to \H^3(W\Omega^2_M) \]
must vanish for the same reason, or because we can split off $\H^0(W\cO_M)$ by restricting to a closed point as explained in \cite[Lem.~5.4(b)]{antieau_bragg}.  We conclude that the descent spectral sequence degenerates at $\E_2$, as claimed.

Alternatively, we could argue that if any differential killed off $\H^3(W\Omega^2_M)$, then would have $\TR_{-1} = W^a$, and in the proof of Proposition \ref{lem:HWofX} below we would find that $\H^2(W\Omega^1_X) = W^a$, which would not give enough torsion in $\H^3(X/W)$ to satisfy the universal coefficient formula
\[ 0 \to \H^2(X/W) \otimes k \to \H^2_{\dR}(X/k) \to \Tor_1^W(\H^3(X/W),k) \to 0 \]
because $h^2_{\dR}(X/k) = 8$.

It remains to analyze the filtration on the $\TR_i$ coming from the descent spectral sequence (see \cite[Lem.~4.3(iv)]{antieau_bragg} for the indexing).  The filtration splits automatically for $i \ne 0$.  If it also splits for $i=0$ then $\TR_0=W^{14}\oplus k$; otherwise $\TR_0=W^{14}$.
\end{proof}

\begin{remark}
It seems reasonable to expect that Ekedahl's duality theorem on the torsion in Hodge--Witt cohomology has an analogue for $\TR$.  In particular, the semisimple torsion in $\TR_i$ should be dual to that in $\TR_{-i-1}$, which in our case would give $\TR_0=W^{14}\oplus k$.
\end{remark}

\begin{proposition} \label{lem:HWofX}
The Hodge--Witt cohomology of $X$ is as follows:
	\[
	\def\arraystretch{1.5}
	\begin{array}
	{c|cccc}
	\H^3 & W      & k            & 0            & W \\
	\H^2 & 0      & W^a \oplus k & W^6 \oplus k & 0 \\
	\H^1 & 0      & W^6          & W^b \oplus k & 0\\
	\H^0 & W      & 0            & 0            & W \\ \hline
	     & W\cO_X & W\Omega^1_X  & W\Omega^2_X  & W\Omega^3_X
	\end{array}
	\]
\end{proposition}
\begin{proof}
As in the proof of Proposition \ref{prop:HWofM}, we need only determine the torsion parts, and there is no torsion in the bottom row, in $\H^1(W\cO_X)$, or in $\H^3(W \Omega^3_X)$.  Moreover, there is no torsion along the diagonals $i+j = 2$ or $i+j=5$ because $b_1(X) = h^1_{\dR}(X)$ and $b_5(X) = h^5_{\dR}(X)$.

The descent spectral sequence for $X$ must degenerate: there cannot be any non-zero differentials on the $\E_2$ page by \cite[Lem.~5.4(2)]{antieau_bragg}, nor on any later page.  This determines all but the torsion in $\H^2(W\Omega^2_X)$, which we find by Ekedahl duality.
\end{proof}

\begin{remark}
Because $h^{2}(\cO_X)=0$ and $h^3(\cO_X)=1$, the Artin--Mazur formal group $\Phi^3=\Phi^3(X,\mathbf{G}_m)$ of $X$ is prorepresentable by a smooth formal group of dimension 1. Its Cartier module is $\H^3(W\cO_X) \cong W$, and it follows that $\Phi^3\cong\widehat{\mathbf{G}}_m$.
\end{remark}

\begin{corollary}
The crystalline cohomology of $X$ and $M$ is as follows:
\[ \def\arraystretch{1.2}
\begin{array}{c|*{2}{C{5ex}}*{3}{C{9ex}}*{2}{C{5ex}}}
i & $0$ & $1$ & $2$ & $3$ & $4$ & $5$ & $6$ \\ \hline
\H^i(X/W) & $W$ & $0$ & $W^6$ & $W^{14} \oplus k^2$ & $W^6 \oplus k^2$ & $0$ & $W$ \\
\H^i(M/W) & $W$ & $0$ & $W^6 \oplus k^2$ & $W^{14}$ & $W^6$ & $k^2$ & $W$
\end{array} \]
\end{corollary}


\appendix
\pagebreak 
\section{A result of Abuaf} \label{app:abuaf}

We give a simplified account of the proof of \cite[Thm.~1.3(4)]{abuaf}.

\begin{theorem}[Abuaf] \label{abuaf_thm}
Let $X$ and $Y$ be smooth complex projective varieties of dimension $\le 4$.  If $D^b(X) \cong D^b(Y)$ then $\H^*(\cO_X) \cong \H^*(\cO_Y)$ as algebras.  In particular $h^{0,j}(X) = h^{0,j}(Y)$ for all $j$.
\end{theorem}

This will follow from two preparatory results:
\begin{proposition} \label{abuaf_prop}
Let $X$ and $Y$ be as in the statement of Theorem \ref{abuaf_thm}, and let $\Phi\colon D^b(X) \to D^b(Y)$ be an equivalence.  Then there is a line bundle $L$ on $X$ such that $\rank(\Phi L) \ne 0$.
\end{proposition}

\begin{lemma} \label{abuaf_lem}
Let $X$ be a smooth complex projective variety, and let
\[ \Hdg^*(X) := \textstyle\bigoplus_p \H^{2p}(X,\bQ) \cap \H^{p,p}(X) \]
be the ring of Hodge classes, endowed with with the Euler pairing
\[ \chi(v,w) = \int_X (v_0 - v_2 + v_4 - \dotsb) \cup (w_0 + w_2 + w_4 + \dotsb) \cup \td(T_X), \]
which makes $\chi(E,F) = \chi(\ch(E),\ch(F))$. \smallskip

If $\dim X \le 3$, then Chern characters of line bundles span $\Hdg^*(X)$. \smallskip

If $\dim X = 4$, then any non-zero $v \in \Hdg^*(X)$ that is (left or right) $\chi$-orthogonal to all Chern characters of line bundles satisfies $\chi(v,v) > 0$.
\end{lemma}

\begin{proof}[Proof of Lemma \ref{abuaf_lem}]
We omit the cases $\dim X \le 1$.

If $\dim X = n \ge 2$ then by the Lefschetz theorem on $(1,1)$-classes we can choose irreducible divisors $D_1, \dotsc, D_k \subset X$ such that their cohomology classes $[D_i]$ span $\Hdg^2(X)$.  Suppose without loss of generality that $D_1$ is ample, so $[D_1]^n = d \cdot [\text{pt}]$ for some $d > 0$.  Observe that
\[ \ch(\cO_{D_i}) = \ch(\cO_X) - \ch(\cO_X(-D_i)) = 0 + [D_i] + \dotsb, \]
and that any product of $\ch(\cO_{D_i})$s is a linear combination of Chern characters of line bundles.

If $\dim X = 2$ then $\Hdg^*(X)$ is spanned by
\begin{align*}
\ch(\cO_X) &= 1 + \dotsb \\
\ch(\cO_{D_i}) &= 0 + [D_i] + \dotsb \\
\ch(\cO_{D_1})^2 &= 0 + 0 + d\cdot[\text{pt}].
\end{align*}

If $\dim X = 3$ then $\Hdg^4$ is spanned by $[D_1].[D_i]$ by the hard Lefschetz theorem, so $\Hdg^*$ is spanned by
\begin{align*}
\ch(\cO_X) &= 1 + \dotsb \\
\ch(\cO_{D_i}) &= 0 + [D_i] + \dotsb \\
\ch(\cO_{D_1}) . \ch(\cO_{D_i}) &= 0 + 0 + [D_1].[D_i] + \dotsb \\
\ch(\cO_{D_1})^3 &= 0 + 0 + 0 + d\cdot[\text{pt}].
\end{align*}

If $\dim X = 4$ then in a similar way we can span
\[ \H^0(X,\bQ) \oplus \H^{1,1}(X,\bQ) \oplus [D_1].\H^{1,1}(X,\bQ) \oplus \H^{3,3}(X,\bQ) \oplus \H^4(X,\bQ). \]
If $v$ is (left or right) $\chi$-orthogonal to this then $v \in \H^{2,2}_\text{prim}(X,\bQ)$, so if $v \ne 0$ then $\chi(v,v) = v_2.v_2 > 0$ by the Hodge--Riemann bilinear relations.
\end{proof}

\begin{proof}[Proof of Proposition \ref{abuaf_prop}]
We have
\[ \rank(\Phi L) = \chi(\Phi L, \cO_y) = \chi(L, \Phi^{-1} \cO_y), \]
where $\cO_y$ is the skyscraper sheaf of some point $y \in Y$.  Suppose this rank is zero for all $L \in \Pic(X)$, and let $v = \ch(\Phi^{-1} \cO_y) \in \Hdg^*(X)$.  The Euler pairing on $\Hdg^*(X)$ is non-degenerate, so if $\dim X \le 3$ then by Lemma \ref{abuaf_lem} we have $v = 0$; but this contradicts the fact that $\chi(\Phi^{-1} \cO_Y, \Phi^{-1} \cO_y) = \chi(\cO_Y, \cO_y) = 1$.  If $\dim X = 4$ then either $v = 0$, which again is impossible, or $\chi(v,v) > 0$, which contradicts the fact that $\chi(\Phi^{-1} \cO_y, \Phi^{-1} \cO_y) = \chi(\cO_y, \cO_y) = 0$.
\end{proof}

\begin{proof}[Proof of Theorem \ref{abuaf_thm}]
Because $\rank(\Phi L) \ne 0$, the natural map of algebra objects
\[ \cO_Y \to \R\cHom_Y(\Phi L,\Phi L) \]
is split by the trace map
\[ \R\cHom_Y(\Phi L,\Phi L) \to \cO_Y, \]
so it induces an injection
\[ \H^*(\cO_Y) \hookrightarrow \Ext^*_Y(\Phi L, \Phi L) = \Ext^*_X(L,L) = \H^*(\cO_X). \]
Symmetrically we get an injection $\H^*(\cO_X) \hookrightarrow \H^*(\cO_Y)$.
\end{proof}

We remark that this proof fails in characteristic $p$ because an equivalence might take all line bundles to objects whose rank is a multiple of $p$.


\section{A higher-dimensional example in any characteristic (by Alexander Petrov)} \label{app:petrov}

Let $p$ be an arbitrary prime number.  Denote $\overline{\mathbb{F}}_p$ by $k$.

\begin{theorem}\label{appendix: main}
There exist smooth projective derived equivalent varieties $X_1, X_2$ over $k$ such that $$h^{0,3}(X_1)\neq h^{0, 3}(X_2)$$ Moreover, for both $i=1,2$ the variety $X_i$ satisfies the following properties:
\begin{enumerate}
\addtolength \itemsep \smallskipamount

\item $X_i$ can be lifted to a smooth formal scheme $\fX_i$ over $W(k)$ such that Hodge cohomology groups $H^r(\fX_i,\Omega^s_{\fX_i/W(k)})$ are torsion-free for all $r,s$.

\item The Hodge-to-de Rham spectral sequence for $X_i$ degenerates at the first page.

\item The crystalline cohomology groups $H^n_{\cris}(X_i/W(k))$ are torsion-free for all $n$.

\item The Hochschild-Kostant-Rosenberg spectral sequence for $X_i$ degenerates at the second page. That is, there exists an isomorphism $\HH_n(X_i/k)\simeq \bigoplus \limits_{s}H^s(X_i,\Omega^{n+s}_{X_i/k})$ for every $n$.

\item $X_i$ cannot be lifted to a smooth algebraic scheme over $W(k)$. 
\end{enumerate}
\end{theorem}

The varieties $X_1,X_2$ are both obtained as approximations of the quotient stack associated to a finite group acting on an abelian variety. The key to the construction is the appropriate choice of such finite group action that relies on complex multiplication and Honda-Tate theory.

Let $G= \bZ/l\bZ$ be the cyclic group of order $l$ where $l$ is an arbitrary odd prime divisor of a number of the form $p^{2r}+1$, for an arbitrary $r\geq 1$.

\begin{proposition}\label{appendix:hodgesym cite}
There exists an abelian variety $A$ over $k$ equipped with an action of $G$ by endomorphisms of $A$ such that \begin{equation}\label{appendix:inequality}\dim_k H^3(A,\cO_A)^{G}\neq \dim_k H^3(\hA,\cO_{\hA})^G\end{equation} Here $\hA$ denotes the dual abelian variety. Moreover, $A$ can be lifted to a formal abelian scheme $\fA$ over $W(k)$ together with an action of $G$.
\end{proposition}

\begin{proof}
Take $A=\fZ\times_{W(k')}k$ with $\fZ$, $k'$ provided by \cite{petrov}, Proposition 3.1. The inequality (\ref{appendix:inequality}) follows because there are $G$-equivariant isomorphisms $H^3(\hA,\cO)\simeq\Lambda^3H^1(\hA,\cO_{\hA})\simeq \Lambda^3 (H^0(A,\Omega^1_{A/k})^{\vee})\simeq H^0(A,\Omega^3_{A/k})^{\vee}$ (the last isomorphism exists even if $p=3$) and $\dim_k H^0(A,\Omega^3_{A/k})^G=\dim_k (H^0(A,\Omega^3_{A/k})^{\vee})^G$ as the order of $G$ is prime to $p$.
\end{proof}

This proposition is specific to positive characteristic. For an abelian variety $B$ equipped with an action of a finite group $\Gamma$ over a field $F$ of characteristic zero there must exist $\Gamma$-equivariant isomorphisms $H^i(B,\Omega^{j}_{B/F})\simeq H^i(\hB,\Omega^j_{\hB/F})^{\vee}$ for all $i, j$ as follows either from Hodge theory or thanks to the existence of a separable $\Gamma$-invariant polarization on $B$. 

A more subtle feature of this construction is that it is impossible to find an abelian variety $B$ with an action of a finite group $\Gamma$ with $p\nmid |\Gamma|$ that would have $\dim_k H^i(B,\cO_B)^{\Gamma}\neq \dim_k H^i(\hB,\cO_{\hB})^{\Gamma}$ for $i=1$ or $i=2$. This can be deduced from Corollary 2.2 of \cite{petrov} applied to an approximation of the stack $[\fB/G]$ where $\fB$ is a formal $\Gamma$-equivariant lift of $B$ that exists by Grothendieck-Messing theory combined with the fact that the order of $\Gamma$ is prime to $p$.

\begin{proof}[Proof of Theorem \ref{appendix: main}] Let $A$ be the abelian variety provided by Proposition \ref{appendix:hodgesym cite}. By Proposition 15 of \cite{serre-mexico} there exists a smooth complete intersection $Y$ of dimension $4$ over $k$ equipped with a free action of $G$. The diagonal action of $G$ on the product of $A\times Y$ is free as well. 

Define $X_1=(A\times Y)/G$ and $X_2=(\hA\times Y)/G$ where $\hA$ is the dual abelian variety of $A$ equipped with the induced action of $G$. In both cases the quotient is taken with respect to the free diagonal action. The equivalence of $D^b(X_1)$ and $D^b(X_2)$ will follow from the Mukai equivalence between derived categories of an abelian scheme and its dual. Indeed, consider $X_1$ and $X_2$ as abelian schemes over $Y/G$. The base changes of both $\Pic^0_{Y/G}(X_1)$ and $X_2$ along $Y\to Y/G$ are isomorphic to $\hA\times Y$ compatibly with the $G$-action. By \'etale descent, $\Pic^0_{Y/G}(X_1)\simeq X_2$ as abelian schemes over $Y/G$. Proposition 6.7 of \cite{bartocci} implies that $D^b(X_1)\simeq D^b(X_2)$.

Next, we compare the Hodge numbers of $X_1$ and $X_2$. By Th\'eor\`eme 1.1 of Expos\'e XI \cite{SGA72} we have $H^i(Y,\cO_Y)=0$ for $1\leq i\leq 3$. Hence, there are $G$-equivariant identifications $H^3(A\times Y,\cO_{A\times Y})\simeq H^3(A,\cO_A)$ and $H^3(\hA\times Y,\cO_{\hA\times Y})\simeq H^3(\hA,\cO_{\hA})$. Since $G$ acts freely on both $A\times Y$ and $\hA\times Y$, the projections $A\times Y\to X_1$ and $\hA\times Y\to X_2$ are \'etale $G$-torsors and, since the order of $G$ is prime to $p$, we have $H^3(X_1,\cO_{X_1})\simeq H^3(A\times Y,\cO_{A\times Y})^G$ and $H^3(X_2,\cO_{X_2})\simeq H^3(\hA\times Y,\cO_{\hA\times Y})^G$.  

The inequality \eqref{appendix:inequality} therefore says that $h^{0,3}(X_1)\neq h^{0,3}(X_2)$. Condition (a) can be fulfilled as it is possible to choose $Y$ that lifts to a smooth projective scheme over $W(k)$ together with an action of $G$, by Proposition 4.2.3 of \cite{raynaud-pictorsion}. Denote by $\fX_1$ and $\fX_2$ the resulting formal schemes over $W(k)$ lifting $X_1$ and $X_2$. Since $\fX_i$ for $i=1,2$ can be presented as a quotient  by a free action of $G$ of a product of an abelian scheme with a complete intersection, the Hodge cohomology modules $H^r(\fX_i,\Omega^s_{\fX_i/W(k)})$ are free for all $r,s$.

Both properties (b) and (d) would be immediate if we had $\dim_k X_i\leq p$ but this is not always possible to achieve. Instead, we can argue using the lifts $\fX_i$. For (b), consider the Hodge-Tate complex $\oPrism_{\fX_i/W(k)[[u]]}$. By Proposition 4.15 of \cite{prisms} there is a morphism $s\colon\Omega^1_{\fX_i/W(k)}[-1]\to \oPrism_{\fX_i/W(k)[[u]]}$ in the derived category of $\fX_i$ that induces an isomorphism on first cohomology. Taking $n$-th tensor power of $s$ and precomposing it with the antisymmetrization map $\Omega^n_{\fX_i/W(k)}\to(\Omega^{1}_{\fX_i/W(k)})^{\otimes n}$ we obtain maps $\Omega^n_{\fX_i/W(k)}[-n]\to \oPrism_{\fX_i/W(k)[[u]]}$ that induce a quasi-isomorphism $\oPrism_{\fX_i/W(k)[[u]]}[\frac{1}{p}]\simeq \bigoplus\limits_{n\geq 0}\Omega^n_{\fX_i/W(k)}[-n]\otimes_{W(k)}W(k)[\frac{1}{p}]$. In particular, the differentials in the Hodge-Tate spectral sequence $H^s(\fX_i,\Omega^{r}_{\fX_i/W(k)})\Rightarrow H^{s+r}_{\oPrism}(\fX_i/W(k)[[u]])$ vanish modulo torsion. But, as we established above, the Hodge cohomology of $\fX_i$ has no torsion, so the Hodge-Tate spectral sequence degenerates at the second page. Therefore the conjugate spectral sequence for $X_i$ degenerates at the second page as well and, equivalently, the Hodge-to-de Rham spectral sequence degenerates at the first page.

Similarly, for (d) consider the Hochschild-Kostant-Rosenberg spectral sequence $E_2^{r,s}=H^r(\fX_i,\Omega^{-s}_{\fX_i/W(k)})$ converging to $\HH_{-r-s}(\fX_i/W(k))$. There exist maps $\varepsilon_n\colon\Omega^n_{\fX_i/W(k)}[n]\to \HH(\fX_i/W(k))$ into the Hochschild complex inducing multiplication by $n!$ on the $n$-th cohomology: $\varepsilon_n=n!\colon\Omega^n_{\fX_i/W(k)}\to \cH^{-n}(\HH(\fX_i/W(k)))\simeq\Omega^n_{\fX_i/W(k)}$. Therefore, the HKR spectral sequence always degenerates modulo torsion, hence degenerates at the second page in our case. Passing to the mod $p$ reduction gives (d).

The property (c) follows from (a) and (b) as $H^n_{\cris}(X_i/W(k))\simeq H^n_{\dR}(\fX_i/W(k))$. 

Finally, to prove (e), note that by the same computation as above one sees that $h^{0,3}(X_1)=h^{3,0}(X_2)\neq h^{0,3}(X_2)=h^{3,0}(X_1)$ so both $X_1$ and $X_2$ violate Hodge symmetry. Denote by $K$ the fraction field of $W(k)$. If $\cX_i$ is a smooth scheme over $W(k)$ lifting $X_i$ then we have \begin{equation}\label{appendix:semicont}
\dim_{K}H^r(\cX_{i,K},\Omega^s_{\cX_{i,K}/K})\leq \dim_k H^r(X_i,\Omega^s_{X_i/k})\end{equation} for all $r,s$ by semi-continuity while $\dim_K H^n_{\dR}(\cX_{i,K}/K)=\dim_k H^n_{\dR}(X_i/k)$ because $H^n_{\dR}(\cX_i/W(k))\simeq H^n_{\cris}(X_i/W(k))$ is torsion-free for all $n$. Since Hodge-to-de Rham spectral sequences for $X_i$ and $\cX_{i,K}$ degenerate at the first page, we deduce that $\sum\limits_{r,s}\dim_{K}H^r(\cX_{i,K},\Omega^s_{\cX_{i,K}/K})=\sum\limits_{r,s}\dim_k H^r(X_i,\Omega^s_{X_i/k})$ so (\ref{appendix:semicont}) is in fact equality for all $r, s$. But this means that the smooth proper algebraic variety $\cX_{i,K}$ over a field of characteristic zero violates Hodge symmetry which is impossible.
\end{proof}

\bibliographystyle{plain}
\bibliography{trimmed}


\end{document}